\pdfoutput=1
\documentclass{tac}

\usepackage{microtype}
\usepackage{macros}
\usepackage[scr]{rsfso}
\usepackage{zi4}
\usepackage{jonclude}

\thirdleveltheorems

\makeatletter

\let\your@endproof\endproof
\def\my@endproof{\your@endproof}
\def\endproof{\my@endproof\gdef\my@endproof{\your@endproof}}
\def\qedhere{\tag*{\endproofbox}\gdef\my@endproof{\relax}}

\title{Strict universes for Grothendieck topoi}
\author{Daniel Gratzer and Michael Shulman and Jonathan Sterling}

\hypersetup{
  pdfauthor = {Daniel Gratzer and Michael Shulman and Jonathan Sterling},
  pdftitle = {Strict universes for Grothendieck topoi},
}

\address{%
  Department of Computer Science, Aarhus University\\
  Department of Mathematics, San Diego University\\
  Department of Computer Science and Technology, University of Cambridge%
}

\eaddress{gratzer@cs.au.dk\CR shulman@sandiego.edu\CR js2878@cl.cam.ac.uk}

\copyrightyear{2022--2024}

\begin{document}

\maketitle
\begin{abstract}
  Hofmann and Streicher famously showed how to lift Grothendieck universes into presheaf topoi, and
  Streicher has extended their result to the case of sheaf topoi by sheafification. In parallel, van
  den Berg and Moerdijk have shown in the context of algebraic set theory that similar constructions
  continue to apply even in weaker metatheories. Unfortunately, sheafification seems not to preserve
  an important \emph{realignment} property enjoyed by presheaf universes that plays a critical
  role in models of univalent type theory as well as synthetic Tait computability.
  When multiple universes are present, realignment also
  implies a \emph{coherent} interpretation of connectives across all universes that justifies the cumulativity laws present in popular
  formulations of Martin-L\"of type theory.

  We observe that a slight adjustment to an argument of Shulman lifts
  a well-behaved cumulative universe hierarchy in the category of sets to a
  cumulative universe hierarchy satisfying the realignment property at every level in any
  Grothendieck topos. Hence one has direct interpretations of Martin-L\"of
  type theory with cumulative universes into all Grothendieck topoi. A further
  implication is to extend the reach of recent synthetic methods in the
  semantics of cubical type theory and the syntactic metatheory of type theory
  and programming languages to all Grothendieck topoi.
\end{abstract}

 \tableofcontents

\section{Introduction}
\label{sec:intro}

Grothendieck introduced the language of \emph{universes} to control the
size issues that plague a na\"ive categorical development of algebraic
geometry~\citep{sga:4}.  In a somewhat different line of research, Martin-L\"of
introduced universes into dependent type theory as a \emph{reflection
principle}~\citep{martin-lof:1971,martin-lof:itt:1975,martin-lof:1979,martin-lof:1984}.
In either case a universe parameterizes a class of maps that are closed under
enough operations to do mathematics, including dependent product/sum,
quotients, \etc.

Grothendieck's use of universes was located in the ambient set theory; each
universe $\mathscr{U}$ determines a category of $\mathscr{U}$-small sets and
functions that serves as a base for both enrichment and internalization,
generalizing the notions of locally small and small category respectively.  The
past three decades have however seen an increased interest in the adaptation of
universes to categories other than $\SET$:

\begin{enumerate}

  \item Universes play a central role in the \emph{algebraic
    set theory} of \citet{joyal-moerdijk:1995}, which explores the relationship between sets and classes from a categorical viewpoint.

  \item Voevodsky's elucidation of the univalence
    principle~\citep{voevodsky:2006}, foreshadowed
    by \citet{hofmann-streicher:1998}, has reinvigorated the study of universes in
    topoi. Closely related to Voevodsky's univalent universes are the
    \emph{object classifiers} of $\infty$-topos theory in the
    Joyal--Lurie--Rezk tradition~\citep{lurie:2009,rezk:2010}.

  \item It is of practical interest to employ Martin-L\"of type theory (MLTT)
    as an internal language for a variety of categories. In
    addition to the standard applications of internal methods to mathematics,
    the existence of topos models of MLTT is a critical ingredient for a
    number of recent results in type theory and programming languages,
    including the generalized abstraction theorem of
    \citet{sterling-harper:2021} and the proofs of normalization for cubical
    type theory and multi-modal dependent type
    theory~\citep{sterling-angiuli:2021,gratzer:2022:lics}.

\end{enumerate}

Unfortunately some doubt has proliferated in the type theoretic literature (\eg
\citet{xu:2015,xu-escardo:2016:universes,coquand-mannaa-ruch:2017}) as to when sufficiently well-adapted
universes exist in a topos. It is a well-known result of
\citet{hofmann-streicher:1997} that Grothendieck universes can be lifted
pointwise into presheaf topoi, and it is slightly less well-known that
sheafification preserves all the properties of this universe that do not
involve strict equality of codes~\citep{streicher:2005,van-den-berg:2011}. Such a sheafified
universe is already sufficient for nearly all mathematical purposes, but falls
short in applications to the semantics and metatheory of dependent type theory,
where certain strict laws not preserved by sheafification remain important.

In this paper we expose an alternative universe construction that applies
in an arbitrary Grothendieck topos, using nothing but the cocompleteness and
exactness properties of Grothendieck topoi. Ours is a variant of a
construction of a universe of presheaves due to \citet{shulman:2015:elegant};
we demonstrate that the resulting universe satisfies an important \emph{realignment} property, which suffices in particular to obtain models of
Martin-L\"of type theory with a cumulative hierarchy of universes in any
Grothendieck topos. The realignment condition is also an important ingredient
in the construction of \emph{univalent} universes for models of homotopy type theory~\citep{hottbook}.

\subsection{Elementary axioms for universes in a topos}

Inspired by the definitions of classes of open and small maps from algebraic
set theory, \citet{streicher:2005} has given a definition of a universe in an elementary
topos $\ECat$ which we review in \cref{def:universe} below.

\begin{notation}
  Given morphisms $\Mor[f]{A}{B}$ and $\Mor[g]{C}{D}$, a morphism $\Mor[\alpha]{f}{g}$ refers to a commuting square from $f$ to $g$:
  \[
    \begin{tikzpicture}[diagram]
      \SpliceDiagramSquare{
        nw = A,
        sw = B,
        ne = C,
        se = D,
        west = f,
        east = g,
        north = \partial_0\alpha,
        south = \partial_1\alpha,
      }
      \node[between = nw and se] {$\alpha$};
    \end{tikzpicture}
  \]

  We shall also freely write $\Mor{f}{g}$ for an \emph{anonymous} square from $f$ to $g$.
\end{notation}

\begin{definition}\label[definition]{def:universe}
  A class of arrows $\Cls\subseteq\Hom[\ECat]$ is called a \emph{universe}
  by \citet{streicher:2005} when it satisfies the following axioms:
  \begin{description}

    \item[(U1)] $\Cls$ is pullback-stable, \ie if $f\in\Cls$ and
      $\Mor{g}{f}$ is a cartesian square, then $g\in\Cls$.

    \item[(U2)] $\Cls$ contains all monomorphisms in $\ECat$.

    \item[(U3)] $\Cls$ is closed under composition.

    \item[(U4)] If $\Mor[f]{A}{I}$ and $\Mor[g]{B}{A}$ are in $\Cls$, then
     the pushforward $\Mor[f_*g]{B}{I}$ lies in $\Cls$.

    \item[(U5)] There exists a \emph{generic} morphism, \ie a morphism
      $\Mor[\pi]{E}{U}\in\Cls$ such that for any $f\in\Cls$ there exists a cartesian map
      $\Mor{f}{\pi}$.

  \end{description}
\end{definition}

The axioms of \cref{def:universe} ensure the closure of $\Cls$ under several
type theoretic operations, if we view an element $\Mor[f]{A}{B}\in\Cls$ as a
dependent type $x:B\vdash A\brk{x}$. Then \textbf{(U1)} corresponds to the
substitution action for dependent types and terms; \textbf{(U2)} states that
all propositions are small; \textbf{(U3-4)} provide for dependent sums and
dependent products, and \textbf{(U5)} provides a generic dependent type
$x:U\vdash E\brk{x}$ of which every other dependent type in $\Cls$ is a
substitution instance.

In the type-theoretic literature, it is the base of this family $U$ which is
called the universe and the generic family is the dependent type $\mathsf{El}$
rendering an element of this universe as a genuine type. We occasionally adopt
this terminology and blur the distinction between a universe and its generic map
by referring to $\Mor{E}{U}$ simply as a universe. Some caution is required:
while a generic map uniquely determines a universe, the converse is not
necessarily true and a universe can have multiple distinct generic maps.

In the context of Martin-L\"of type theory, it is common to study classes of
maps that may not satisfy all the axioms above; for instance, type theory is
often used in settings that do not have a single well-behaved notion of
proposition, so \textbf{(U2)} loses some significance.  We therefore define a
notion of \emph{pre-universe} below.

\begin{definition}\label[definition]{def:pre-universe}
  A \emph{pre-universe} is a class of arrows satisfying axioms \textbf{(U1,U3--5)}.
\end{definition}

\citet{streicher:2005} discusses some additional useful but optional axioms for universes.
\begin{description}
  \item[(U6)] (Propositional subuniverse) $\Cls$ contains the terminal map $\Mor{\Omega}{\TermObj{\ECat}}$.\footnote{\citet{streicher:2005} refers to this property as impredicativity, but we wish to avoid confusion with a different notion of impredicativity that involves the existence of dependent products along maps \emph{not} in $\Cls$, which has its prototype in the full internal subcategory of the category of assemblies spanned by modest sets~\citep{streicher:2017-2018,hyland:1988,hyland-robinson-rosolini:1990}.}
  \item[(U7)] (Descent) If $g\in\Cls$ and $\Mor|->>|{g}{f}$ is a cartesian epimorphism, then $f\in\Cls$.
\end{description}

A Grothendieck universe $\VTY$ in $\SET$ is readily seen to induce a universe $\Cls_\VTY$ in the
sense of \cref{def:universe} where $\Cls_\VTY$ consists of the collection of maps with $\VTY$-small
fibers.  \citet{hofmann-streicher:1997} and \citet{streicher:2005} have shown that $\Cls_\VTY$ can
be lifted systematically to presheaves and sheaves.
The first result in particular has been widely used in the semantics of type theory, because the
generic morphism satisfies a number of strict equations specific to its construction. These
additional equations are crucial for modeling \eg{} strict cumulative universes. Other more novel
applications of this strictness have emerged in models of Voevodsky's univalence axiom and homotopy
type theory. Only more recently has an axiomatic basis for these stricter Hofmann--Streicher
universes been isolated:

\begin{definition}\label[definition]{def:realignment}
  A universe $\Cls$ is said to have \emph{realignment} with respect to a class $\mathcal{M}$ of monomorphisms when
  axiom \textbf{(U8)} below is satisfied:\footnote{Our axiom \textbf{(U8)} is denoted (2$'$) by \citet{shulman:2015:elegant}.}
  \begin{description}
    \item[\textcolor{RegalBlue}{(U8)}]
      A chosen cartesian morphism $\Mor{h}{\pi}$ into the generic morphism can
      be extended along any cartesian monomorphism $\Mor|>->|{h}{f}$ lying
      horizontally over an element of $\mathcal{M}$ where $f\in\Cls$:
      \[
        \begin{tikzpicture}[diagram]
          \node (nw) {$h$};
          \node (sw) [below = 2.5cm of nw] {$f$};
          \node (ne) [right = 2.5cm of nw] {$\pi$};
          \path[>->] (nw) edge node [sloped,below] {cart.} (sw);
          \path[->] (nw) edge node [above] {cart.} (ne);
          \path[->,exists] (sw) edge node [sloped,below] {cart.} (ne);
        \end{tikzpicture}
      \]
  \end{description}

  Unless otherwise specified, $\mathcal{M}$ is the class of all monomorphisms.
\end{definition}

\begin{remark}
  \label[remark]{rem:other-appearances}
  While \citet{shulman:2015:elegant} extracted \textbf{(U8)} from the construction of the universal
  Kan fibration given by Kapulkin, Lumsdaine, and Voevodsky~\citep{kapulkin-lumsdaine:2021}, similar properties have since appeared in
  the construction of the universal left fibration~\citep[Corollary 5.2.6]{cisinski:2019} and the
  universal cocartesian fibration~\citep[\href{https://kerodon.net/tag/0293}{Tag 0293}]{kerodon}.
\end{remark}

\begin{remark}
  \label[remark]{rem:unfolded-realignment}
  Unfolding the fibrational language, \cref{def:realignment} can be stated more explicitly.
  We require that given
  $\Mor|>->|[m]{A}{B}\in\mathcal{M}$ and $\Mor[f]{Q}{B}\in\mathcal{S}$, any
  cartesian square $\Mor{m^*f}{\pi}$ extends along $m$ to a cartesian square
  $\Mor{f}{\pi}$:
  \[
    \begin{tikzpicture}[diagram,baseline = (sq/sw.base)]
      \SpliceDiagramSquare<sq/>{
        width = 3.5cm,
        height = 2.5cm,
        nw = m^*Q,
        sw = A,
        ne = E,
        se = U,
        east = \pi,
        west = m^*f,
        south/style = {color = {LightGray}},
        south/node/style = {xshift = 0.75cm, above},
      }
      \node (n) [between = sq/nw and sq/ne, yshift = -0.75cm] {$Q$};
      \node (s) [between = sq/sw and sq/se, yshift = -0.75cm] {$B$};
      \path[->] (n) edge node [upright desc] {$f$} (s);
      \path[>->] (sq/nw) edge (n);
      \path[>->] (sq/sw) edge node [sloped,below] {$m$} (s);
      \path[exists,->] (s) edge (sq/se);
      \path[exists,->] (n) edge (sq/ne);
    \end{tikzpicture}
  \]
\end{remark}

Intuitively, \textbf{(U8)} extends \textbf{(U5)} to provide a more refined
generic map where a representation $\Mor{f}{\pi}$ of an arrow $f \in \Cls$ can
be chosen to strictly extend a representation of $g$ where
$\Mor|>->|{g}{f} \in \mathcal{M}$. In practice, one often exhibits a
representation $\Mor{f}{\El}$ to show $f \in \Cls$ only to discard this square
to obtain a \emph{realigned} representation of $f$ which coheres with a
previously chosen representation of $\Mor|>->|{g}{f}$ using \textbf{(U8)}.

We note that \textbf{(U8)} subsumes \textbf{(U5)} under appropriate conditions on $\mathcal{M}$.
\begin{lemma}\label[lemma]{lem:genericity-from-realignment}
  Suppose $\Cls$ is a pullback-stable class of maps and $\pi\in\Cls$ is a morphism
  satisfying \textbf{(U8)} with $\mathcal{M}$ containing all maps of the form $\Mor{\InitObj{\ArrCat{\ECat}}}{f}$, where $\InitObj{\ArrCat{\ECat}}$ is the identity map on $\InitObj{\ECat}$; then the pair $\prn{\Cls,\pi}$ satisfies \textbf{(U5)}.
\end{lemma}

\begin{proof}
  Fixing an element $f\in\Cls$, we must construct a cartesian morphism
  $\Mor{f}{\pi}$; this is achieved by realigning
  $\Mor{\InitObj{\ArrCat{\ECat}}}{\pi}$ along
  $\Mor|>->|{\InitObj{\ArrCat{\ECat}}}{f}$:
  \[
    \begin{tikzpicture}[diagram,baseline=(f.base)]
      \node (h) {$\InitObj{\ArrCat{\ECat}}$};
      \node (f) [below = of h] {$f$};
      \node (pi) [right = of h] {$\pi$};
      \draw[->] (h) to (pi);
      \draw[>->] (h) to (f);
      \draw[->,exists] (f) to (pi);
    \end{tikzpicture}
    \qedhere
  \]
\end{proof}

\subsection{From realignment to cumulative hierarchies}\label{sec:cumulative-hierarchy}

The true utility of \textbf{(U8)} is the ability to choose a representation for a morphism
$f \in \Cls$ subject to a strict equation. For instance, \textbf{(U8)} is sufficient to `strictify'
a hierarchy of universes so that the choices of codes for connectives commute with the coercion maps from
one universe to another~\citep{shulman:2015:elegant}. In particular, let $\Cls\subseteq\mathcal{T}$
be two universes equipped with a choice of cartesian monomorphism
$\Mor|>->|[i]{\pi\Sub{\Cls}}{\pi\Sub{\mathcal{T}}}$. Further assume that $\mathcal{T}$ satisfies
realignment for the class of all monomorphisms.

\begin{notation}
  Given a morphism $\Mor[f]{X}{Y}$, we write $\Mor[P\Sub{f}]{\ECat}{\ECat}$ for the
  polynomial endofunctor given by the composite $Y_! \circ f_* \circ X^*$.
\end{notation} 

Both $\Cls,\mathcal{T}$ are closed under dependent products, hence there exist cartesian
morphisms
$\Mor[\Pi\Sub{\Cls}]{P\Sub{\pi\Sub{\Cls}}\prn{\pi\Sub{\Cls}}}{\pi\Sub{\Cls}}$
and
$\Mor[\Pi\Sub{\mathcal{T}}]{P\Sub{\pi\Sub{\mathcal{T}}}\prn{\pi\Sub{\mathcal{T}}}}{\pi\Sub{\mathcal{T}}}$,
but \cref{diag:pi-lift-fail} below need not commute:
\begin{equation}\label[diagram]{diag:pi-lift-fail}
  \begin{tikzpicture}[diagram,baseline=(sw.base)]
    \SpliceDiagramSquare{
      nw = P\Sub{\pi\Sub{\Cls}}\prn{\pi\Sub{\Cls}},
      sw = P\Sub{\pi\Sub{\mathcal{T}}}\prn{\pi\Sub{\mathcal{T}}},
      ne = \pi\Sub{\Cls},
      se = \pi\Sub{\mathcal{T}},
      west = P\Sub{i}\prn{i},
      east = i,
      north = \Pi\Sub{\Cls},
      south = \Pi\Sub{\mathcal{T}},
      west/style = >->,
      east/style = >->,
      width = 2.5cm,
    }
    \node [between = nw and se] {$\puncture$};
  \end{tikzpicture}
\end{equation}

We can replace $\Pi\Sub{\Cls},\Pi\Sub{\mathcal{T}}$ with new codes
$\Pi\Sub{\Cls}',\Pi\Sub{\mathcal{T}}'$ for which the analogue to
\cref{diag:pi-lift-fail} commutes. We set $\Pi\Sub{\Cls}' \coloneqq
\Pi\Sub{\Cls}$ and define $\Pi\Sub{\mathcal{T}}'$ by realigning
$i\circ\Pi'\Sub{\Cls}$ along $P_i\prn{i}$:
\begin{equation}
  \DiagramSquare{
    nw = P\Sub{\pi\Sub{\Cls}}\prn{\pi\Sub{\Cls}},
    sw = P\Sub{\pi\Sub{\mathcal{T}}}\prn{\pi\Sub{\mathcal{T}}},
    ne = \pi\Sub{\Cls},
    se = \pi\Sub{\mathcal{T}},
    west = P\Sub{i}\prn{i},
    east = i,
    north = \Pi\Sub{\Cls}',
    south = \Pi\Sub{\mathcal{T}}',
    west/style = >->,
    east/style = >->,
    south/style = exists,
    width = 2.5cm,
  }
\end{equation}

If we further assume that $\ECat$ is sufficiently cocomplete, \eg{}, if it is a
Grothendieck topos, the technique above easily extends to infinite and even
transfinite hierarchies of universes. In the latter case, one realigns along the
\emph{join} of all the subobjects
$\Mor|>->|{P\Sub{\pi\Sub{\Cls'_i}}\prn{\pi\Sub{\Cls'_i}}}{P\Sub{\pi\Sub{\mathcal{T}}}\prn{\pi\Sub{\mathcal{T}}}}$
pertaining to the formation data for dependent product type codes at lower
universes. Then a coherent hierarchy of such codes is built ``from the ground
up'' by induction.

\subsection{Structure of the paper}

We survey the landscape of universe constructions available in Grothendieck toposes and show that
they inherit a plentiful supply of well-behaved universes from $\SET$.

\textbf{\cref{sec:hs-and-s}}.  We revisit the presheaf-theoretic universe
construction of \citet{hofmann-streicher:1997}, lifting a Grothendieck universe
in $\SET$ to a universe of pointwise small families of presheaves satisfying
\textbf{(U1--8)}. Presenting a sheaf topos as a subcategory of a presheaf
topos, we recall from \citet{streicher:2005} that the Hofmann--Streicher
construction also produces universes of sheaves satisfying \textbf{(U1--6)},
as the sheafification of the generic small family of presheaves is generic for
small families of sheaves.

\textbf{\cref{sec:descent}}.  We review a number of categorical preliminaries
to our main result involving descent and $\kappa$-compactness.

\textbf{\cref{sec:universe}}. Adapting a construction of
\citet{shulman:2015:elegant}, we prove our main result
(\cref{cor:rel-cpt-universe-all-axioms}): the universe of relatively
$\kappa$-compact sheaves for a strongly inaccessible cardinal $\kappa$
satisfies all the universe axioms including \textbf{(U8)}. We deduce that
cumulative hierarchies of strict universes lift from $\SET$ to any Grothendieck
topos.

\textbf{\cref{sec:internal-formulations}}. We discuss and compare two
equivalent formulations of the realignment property employing the internal
language of a topos.

\textbf{\cref{sec:examples}}. The results of \cref{sec:universe} have important
consequences for the syntax and semantics of type theory; we review a few of
these applications in \cref{sec:examples}. For instance, we have already shown
that \textbf{(U8)} is sufficient to construct strictly cumulative hierarchies
of universes, and with the existence of these hierarchies in arbitrary
Grothendieck topoi the independence of several logical principles of
Martin-L{\"o}f type theory immediately follows; contrary to some claims, sheaf
semantics is sufficient and there is no need to move from sheaves to stacks. We
outline applications to independence results in \cref{sec:independence}.

We also illustrate the general utility of \textbf{(U8)} through two specific
examples: the semantics of univalence in homotopy type theory (\cref{sec:univalence})
and the construction of glued models of type theory
(\cref{sec:realignment-stc}) for proving syntactic metatheorems such as
canonicity, normalization, and decidability. In both cases, \textbf{(U8)}
allows us to leverage existing categorical machinery while still maintaining
the required strict equations.

\subsubsection*{Foundational assumptions.}

Throughout, we work in a sufficiently strong metatheory to ensure that $\SET$
comes equipped with a collection of universes \eg{}, ZFC with the Grothendieck
universe axiom; we make use of the axiom of choice. We return to this topic
briefly in \cref{sec:conclusion:constructive}.

\paragraph{Acknowledgments}

We are grateful to Steve Awodey, Thomas Streicher, and the anonymous referees
for helpful feedback and corrections to an earlier draft of this paper. This
research was supported by the United States Air Force Office of Scientific
Research under award numbers FA9550-21-1-0009 and FA9550-23-1-0728 (Tristan
Nguyen, program officer).

\section{Reviewing Hofmann and Streicher's universes}
\label{sec:hs-and-s}

We begin by recalling constructions from \citet{hofmann-streicher:1997} and \citet{streicher:2005}
lifting universes from $\SET$ to Grothendieck topoi. To begin with, fix a \emph{Grothendieck
  universe} $\VTY$, a transitive non-empty set closed under Kuratowski pairing,
power-sets, and $I$-indexed unions for each $I \in \VTY$.

\subsection{Universes of sets}

Each Grothendieck universe defines a universe as in \cref{def:universe}.

\begin{construction}
  Define the universe $\Cls\Sub{\VTY} \subseteq \Hom[\SET]$ to be the collection of all morphisms
  $\Mor[f]{X}{Y}$ with $\VTY$-small fibers: explicitly for each $y \in Y$, there exists a
  $u \in \VTY$ such that $u \cong f^{-1}\prn{y}$.
\end{construction}

Showing that $\Cls\Sub{\VTY}$ satisfies axioms \textbf{(U1--4,6,7)} is a
standard exercise. Setting $\VEL = \Sum{u : \VTY} u$, the generic map is given
by the projection $\Mor[\VEl]{\VEL}{\VTY}$. The proof that $\VEl$ is generic mostly
unsurprising but we note that the axiom of choice is required---essentially to
produce an assignment of $\VTY$ representatives for the fibers of a morphism in
$\Cls\Sub{\VTY}$ which are known only to be isomorphic to elements of $\VTY$.

\begin{lemma}
  \label[lemma]{lem:grothendieck-uni-realignment}
  The universe $\Cls\Sub{\VTY}$ satisfies the realignment axiom \textbf{(U8)}.
\end{lemma}
\begin{proof}
  Recalling the characterization of \textbf{(U8)} given by
  \cref{rem:unfolded-realignment}, we fix a realignment problem of the following
  form:
  \[
    \begin{tikzpicture}[diagram,baseline = (sq/sw.base)]
      \SpliceDiagramSquare<sq/>{
        width = 3.5cm,
        height = 2.5cm,
        nw = m^*Q,
        sw = A,
        ne = {\Sum{u : \VTY} u},
        se = \VTY,
        east = \VEl,
        west = m^*f,
        south = {p},
        south/style = {color = {LightGray}},
        south/node/style = {xshift = 0.75cm, above},
      }
      \node (n) [between = sq/nw and sq/ne, yshift = -0.75cm] {$Q$};
      \node (s) [between = sq/sw and sq/se, yshift = -0.75cm] {$B$};
      \path[->] (n) edge node [upright desc] {$f$} (s);
      \path[>->] (sq/nw) edge (n);
      \path[>->] (sq/sw) edge node [sloped,below] {$m$} (s);
      \path[exists,->] (s) edge node[below] {$q$} (sq/se);
      \path[exists,->] (n) edge (sq/ne);
    \end{tikzpicture}
  \]

  Suppose further that $f \in \Cls\Sub{\VTY}$ and, through \textbf{(U5)}, pick some morphism
  $\Mor[q_0]{B}{\VTY}$ classifying $f$. While $q_0$ does not necessarily fit into the above diagram,
  we use it to define a map $\Mor[q]{B}{\VTY}$ that does:
  \[
    q(b) =
    \begin{cases}
      p(a) & \text{when } b = m(a)
      \\
      q_0(b) & \text{otherwise}
    \end{cases}
  \]

  This definition is well-defined as $m$ is a monomorphism; there is at most one $a$ such that
  $m(a) = b$. By definition $q$ fits into the triangle above, and an identical procedure extends it
  to the required cartesian square $\Mor{f}{\VEl}$.
\end{proof}

\begin{remark}
  The above proof can be generalized to show that any universe in a boolean
  topos satisfying \textbf{(U5)} satisfies \textbf{(U8)}.
\end{remark}

\begin{remark}
  In the category of sets, any universe in the sense of the present axioms
  determines a universe in the sense of Grothendieck. Streicher's axioms for
  universes can therefore be thought of as a more \emph{direct} alternative to
  Grothendieck's axioms, emphasizing ordinary mathematical constructions (\eg
  dependent product, sum, quotient) rather than set theoretical considerations
  (transitive membership, power sets, \etc).
\end{remark}

\subsection{Hofmann and Streicher's universe of presheaves}

Given a $\VTY$-small category $\CCat$, the universe $\Cls\Sub{\VTY}$ induces a suitable universe $\PrCls\Sub{\VTY}$ on
$\Psh{\CCat}$ that we explore below.

\begin{definition}
  We define $\PrCls\Sub{\VTY}$ to consist of morphisms $\Mor[f]{X}{Y}$ such that for each
  cartesian square of the following shape, the presheaf $y^*X$ is (essentially) $\VTY$-valued:
  \[
    \DiagramSquare{
      nw = y^*X,
      sw = \Yo{C},
      ne = X,
      se = Y,
      nw/style = pullback,
      south = y,
      east = f,
    }
  \]

  Explicitly, for each $D : \CCat$ the set $\prn{y^*X}_D$ must be $\VTY$-small.
\end{definition}

\begin{remark}
  We may equivalently describe $\PrCls\Sub{\VTY}$ as the class of maps
  $\Mor[f]{X}{Y}$ such that the fibers of $f$ over representables are
  $\VTY$-small.
\end{remark}

Again, it remains to show that this class satisfies the expected axioms. \textbf{(U1--4,6,7)} follow
through calculation (taking advantage of the standard construction of $f_*g$ for \textbf{(U4)} and
$\Omega$ for \textbf{(U6)}). \citet{hofmann-streicher:1997} show that
$\PrCls\Sub{\VTY}$ satisfies \textbf{(U5)} with a generic map
$\Mor[\El]{\EL}{\TY}$. The construction of $\El$ is highly dependent on $\Psh{\CCat}$ being a
presheaf category, taking advantage of the correspondence
$\Sl{\Psh{\CCat}}{\Yo{C}} \simeq \Psh{\Sl{\CCat}{C}}$ which represents the codomain fibration as a
strict 2-functor rather than the usual pseudofunctor. This correspondence restricts
to presheaves valued in the full subcategory of $\SET$ spanned by elements of
$\VTY$ to induce an equivalence $\Sl{\Psh[\VTY]{\CCat}}{\Yo{C}} \simeq
\Psh[\VTY]{\Sl{\CCat}{C}}$. We use this to define $\TY_C$ as follows:
\[
  \TY_C = \Psh[\VTY]{\Sl{\CCat}{C}}
\]

The generic family $\El$ is most directly defined as a presheaf over $\ElCat{\TY}$, again taking
advantage of the equivalence $\Sl{\Psh{\CCat}}{\TY} \simeq \Psh{\ElCat{\TY}}$
\[
  \El_{\prn{C,X}} = X_{\prn{C,\IdArr{}}}
\]

The following is a result of \textcite{hofmann-streicher:1997}.

\begin{theorem}
  $\El$ satisfies \textbf{(U5)}.
\end{theorem}
\begin{proof}
  Fix a map $\Mor[f]{Q}{X} \in \PrCls\Sub{\VTY}$. We must show that there exists
  some cartesian square $\Mor{f}{\El}$. First, let us note that $\Mor[f]{Q}{X}$
  induces a presheaf $F : \Psh{\ElCat{X}}$ and our assumption that
  $f \in \PrCls\Sub{\VTY}$ ensures that $F$ is essentially $\VTY$-small. In
  particular, we may choose $F' \cong F$ such that $F'$ belongs to the
  subcategory $\Psh[\VTY]{\ElCat{X}}$.

  We will now construct a cartesian square $\Mor{f}{\El}$ by defining a morphism
  explicitly $\Mor[q]{X}{\TY}$ and then argue that $q^*\El = f$. To this end,
  let us fix $C : \CCat$ along with $x \in X\Sub{C}$ and define
  $q_C\prn{x} \in \TY_C = \Psh[\VTY]{\Sl{\CCat}{C}}$:
  \[
    q_C\prn{x}\Sub{\prn{D,c}} = F'\prn{D,x \cdot c}
  \]
  The computation that $q$ organizes into a natural transformation is
  routine.

  It remains only to argue that $q^*\El$ is isomorphic to $f$. Examining the
  definition of $\El$, it is easiest to argue this by once more passing to
  $\Psh{\ElCat{X}}$ and showing that $q^*\El \cong F$. However, by definition
  $q^*\El$ is isomorphic to $F'$ which is in turn isomorphic to $F$.
\end{proof}

The generic map $\El$ satisfies a number of strict equations and, in particular, it satisfies
\textbf{(U8)}. The proof is similar to \cref{lem:grothendieck-uni-realignment}, but the additional
indexing over $\CCat$ obscures this similarity. Accordingly, we introduce a small amount of
machinery beforehand.

Observe first that we may view both $\VTY$ and $\VEL$ as categories,
respectively the categories of $\VTY$-sets and pointed $\VTY$-sets; this
viewpoint is exposed in detail by
\citet[Section~1]{awodey-gambino-hazratpour:2021}. Given that both $\VTY$ and
$\VEL$ are small, we may view them as categories internal to $\SET$. For formal
reasons, the projection $\Mor[\VEl]{\VEL}{\VTY}$ is then a category internal to
$\ArrCat{\SET}$. From this perspective, each
$\El_C = \Mor{\Psh[\VEL]{\Sl{\CCat}{C}}}{\Psh[\VTY]{\Sl{\CCat}{C}}} :
\ArrCat{\SET}$ (the component of the presheaf morphism $\El$ at $C:\CCat$) is
precisely the objects of the category $\VEl$-valued presheaves on
$\Mor[\IdArr{}]{\Sl{\CCat}{C}}{\Sl{\CCat}{C}}$ internal to $\ArrCat{\SET}$.

Next, let $\Mor[\alpha]{f}{\El}$ be a cartesian map in $\ArrCat{\Psh{\CCat}}$; there is a canonical
cartesian map $\Mor[\hat\alpha_C]{f_C}{\VEl}$ in $\ArrCat{\SET}$ defined like so:
\[
  \hat\alpha_C\prn{x} = \alpha_C\prn{x}\prn{\IdArr{C}}
\]

Returning to the perspective of $\ArrCat{\SET}$, the element $\alpha_C\prn{x}$
is a $\VEl$-valued presheaf on $\Sl{\CCat}{C}$, hence evaluating at $\IdArr{C}$
yields an element of $\VEl$.

\begin{theorem}
  The universe $\PrCls\Sub{\VTY}$ satisfies realignment \textbf{(U8)}.
\end{theorem}

\begin{proof}
  Fix a realignment problem of the following form in which $\phi$ and $\alpha$ are cartesian, and
  there exists some cartesian map $\Mor[\chi]{f}{\El}$ that we wish to realign as the dotted lift
  depicted below:
  \begin{equation}\label[diagram]{diag:lift-glue:0}
    \begin{tikzpicture}[diagram,baseline=(sw.base)]
      \node (nw) {$h$};
      \node (sw) [below = 2.5cm of nw] {$f$};
      \node (ne) [right = 2.5cm of nw] {$\El$};
      \path[>->] (nw) edge node [left] {$\phi$} (sw);
      \path[->] (nw) edge node [above] {$\alpha$} (ne);
      \path[->,exists] (sw) edge node [sloped,below] {cart.} (ne);
    \end{tikzpicture}
  \end{equation}

  For each $C : \CCat$, we transform the above into a realignment problem for the universe
  $\Mor[\VEl]{\VEL}{\VTY}$ of sets in terms of the cartesian map
  $\Mor[\hat\alpha_C]{h_C}{\VEl}$. This yields a cartesian lift $\Mor[\beta_C]{f_C}{\VEl}$ in the
  following configuration.
  \begin{equation}\label[diagram]{diag:lift-glue:1}
    \begin{tikzpicture}[diagram,baseline=(sw.base)]
      \node (nw) {$h_C$};
      \node (sw) [below = 2.5cm of nw] {$f_C$};
      \node (ne) [right = 2.5cm of nw] {$\VEl$};
      \path[>->] (nw) edge node [left] {$\phi_C$} (sw);
      \path[->] (nw) edge node [above] {$\hat\alpha_C$} (ne);
      \path[->,exists] (sw) edge node [sloped,below] {$\beta_C$} (ne);
    \end{tikzpicture}
  \end{equation}

  The above is possible because $f_C$ is classified by $\VEl$. Hence we may define a natural
  transformation $\Mor[\check{\beta}]{f}{\El}$ fitting into \cref{diag:lift-glue:0} as follows:
  \[
    \check{\beta}_C\prn{x}\prn{\Mor[z]{D}{C}} = \beta_D\prn{z\cdot x}
  \]

  The functorial action on morphisms of $\Mor{z'}{z} : \Sl{\CCat}{C}$ is
  obtained from the fact that each $\beta_D\prn{z\cdot x}$ is isomorphic to
  $\chi_D\prn{z\cdot x}\prn{\IdArr{D}}$, which is a fiber of a $\VEl$-valued
  presheaf and hence has the needed functorial action.
  To check that $\check{\beta}$ restricts along $\phi$ to $\alpha$, we fix
  $\Mor[z]{D}{C}$ and compute:
  \begin{align*}
    \check{\beta}_C\prn{\phi_C\prn{x}}\prn{z}
    &= \beta_D\prn{z\cdot\phi_D\prn{x}}
    \\
    &= \beta_D\prn{\phi_C\prn{z\cdot x}}
    \\
    &= \hat\alpha_D\prn{z\cdot x}
    \\
    &= \alpha_D\prn{z\cdot x}\prn{\IdArr{D}}
    \\
    &= \alpha_C\prn{x}\prn{z}
      \qedhere
  \end{align*}
\end{proof}

\begin{theorem}
  \label[theorem]{thm:hs}
  The class of morphisms $\PrCls\Sub{\VTY}$ in $\Psh{\CCat}$ is a universe satisfying
  \textbf{(U1--8)}.
\end{theorem}

\subsection{Streicher's universe of sheaves}

Fixing a Grothendieck topology $J$ on $\CCat$, we show that the universe $\PrCls\Sub{\VTY}$ induces a
universe on $\Sh{\CCat,J}$.
Let $i : \Sh{\CCat,J} \to \Psh{\CCat}$ denote the inclusion geometric morphism, so that $i_*$ is the inclusion functor and $i^*$ is sheafification.
\begin{definition}
  We define $\ShCls\Sub{\VTY}$ to be the collection of all maps $f$ such that
  $i_*f \in \PrCls\Sub{\VTY}$.
\end{definition}

This collection of maps is easily shown to satisfy \textbf{(U1--4)} because
$i_*$ preserves finite limits. The existence of a generic map \textbf{(U5)} has
been the source of controversy within the type-theoretic literature; one
potential candidate is the restriction of $\pi\Sub{\PrCls\Sub{\VTY}}$ to the presheaf of pointwise $\VTY$-small
sheaves, but this is not actually a sheaf as pointed out by
\citet{xu-escardo:2016:universes}.
\citet{streicher:2005} proposed a more direct approach: the generic map for
$\ShCls\Sub{\VTY}$ is the sheafification of the generic map for
$\PrCls\Sub{\VTY}$. To prove this, we recall Proposition~5.4.4 of \citet{van-den-berg:habil}:

\begin{proposition}
  \label[proposition]{lem:sheafification-preserves-smallness}
  If $f \in \PrCls\Sub{\VTY}$ then $i^*f \in \ShCls\Sub{\VTY}$.
\end{proposition}

With this to hand, we immediately conclude that $i^*\El \in \ShCls\Sub{\VTY}$.
\begin{corollary}
  The family $i^*\El$ is generic for $\ShCls\Sub{\VTY}$.
\end{corollary}
\begin{proof}
  Fix $\Mor[f]{X}{Y} \in \ShCls\Sub{\VTY}$. By definition,
  $i_*f \in \PrCls\Sub{\VTY}$ so by \textbf{(U5)} the following cartesian
  square exists:
  \begin{equation}
    \DiagramSquare{
      nw = {i_*X},
      sw = {i_*Y},
      ne = {\EL},
      se = {\TY},
      nw/style = pullback,
    }
  \end{equation}

  The image of this cartesian square under $i^*$ remains cartesian and thus shows that
  $f \cong i^*i_*f$ is classified by $i^*\El$.
\end{proof}

\begin{theorem}
  \label[theorem]{thm:s}
  The class of maps $\ShCls\Sub{\VTY}$ is a universe satisfying \textbf{(U1--6)}.
\end{theorem}

It is natural to wonder whether this universe satisfies \textbf{(U8)}, but unfortunately this does
not seem to be the case. Fix a realignment problem in $\Sh{\CCat}$:
\begin{equation}\label[diagram]{diag:sh-glue-problem-0}
  \begin{tikzpicture}[diagram,baseline=(sw.base)]
    \node (nw) {$h$};
    \node (sw) [below = 2.5cm of nw] {$f$};
    \node (ne) [right = 2.5cm of nw] {$i^*{\El}$};
    \path[>->] (nw) edge node [left] {$\phi$} (sw);
    \path[->] (nw) edge node [above] {$\alpha$} (ne);
    \path[->,exists] (sw) edge node [sloped,below] {cart.} (ne);
  \end{tikzpicture}
\end{equation}

By definition, $i_*{f}$ and $i_*{h}$ both belong to
$\PrCls\Sub{\VTY}$. Moreover, since $i_*i^*{\El} \in \PrCls\Sub{\VTY}$ we obtain
a cartesian morphism $\Mor[u]{i_*i^*{\El}}{\El}$ and so
\cref{diag:sh-glue-problem-0} induces a realignment problem in $\Psh{\CCat}$
that can then be solved:
\begin{equation}\label[diagram]{diag:sh-glue-problem-1}
  \begin{tikzpicture}[diagram,baseline=(sw.base)]
    \node (nw) {$i_*{h}$};
    \node (sw) [below = 2.5cm of nw] {$i_*{f}$};
    \node (ne) [right = 2.5cm of nw] {$i_*i^*{\El}$};
    \node (nee) [right = 2.5cm of ne] {$\El$};
    \path[>->] (nw) edge node [left] {$i_*\phi$} (sw);
    \path[->] (nw) edge node [above] {$i_*\alpha$} (ne);
    \path[->] (ne) edge node [above] {$u$} (nee);
    \path[->,exists] (sw) edge node [sloped,below] {$\beta$} (nee);
  \end{tikzpicture}
\end{equation}

While this appears promising, there is no obvious way to relate this realignment problem in $\El$
to a solution in $i^*{\El}$. In particular, $i^*{u}$ is not the counit
$\Mor[\epsilon]{i^*i_*i^*{\El}}{i^*{\El}}$ so $i^*{\beta} \circ \epsilon^{-1}$ does not
satisfy the correct boundary condition.

Indeed, one can produce counterexamples to the claim. We are indebted to the
reviewer who suggested the following counterexample.

\begin{lemma}
  There exists a $\VTY$-small site $\prn{\CCat,J}$ such that $i^*\El$ does not
  satisfy \textbf{(U8)}.
\end{lemma}
\begin{proof}
  Define $\CCat = \brc{0 \le 1} \times \brc{0 \le 1}$ and let $J$ be such
  that $\prn{0,1}$, $\prn{1,0}$, and $\prn{1,1}$ have no non-trivial covers
  while $\prn{0,0}$ is covered by the empty sieve. The sheafification functor
  $\Mor[i^*]{\Psh{\CCat}}{\Sh{\CCat,J}}$ sends a presheaf
  $\Mor[X]{\OpCat{\CCat}}{\SET}$ to the following sheaf:
  \begin{alignat*}{2}
    \prn{i^*X}_{\prn{0,0}} &= \TermObj{} &\qquad \prn{i^*X}_{\prn{0,1}} &= X_{\prn{0,1}}
    \\
    \prn{i^*X}_{\prn{1,0}} &= X_{\prn{1,0}} & \prn{i^*X}_{\prn{1,1}} &= X_{\prn{1,1}}
  \end{alignat*}

  In particular, both ${i^*\TY}_{0,1}$ and ${i^*\TY}_{0,1}$ are isomorphic to
  $\Ob\prn{\ArrCat{\VTY}}$. Let us consider the arrows
  $\Mor{\InitObj{}}{\TermObj{}}$ and $\Mor{\TermObj{}}{\mathbf{2}}$ in $\VTY$
  and write $\Mor[f_{01}]{\Yo{0,1}}{\TY}$ for the map induced by the former and
  $\Mor[f_{10}]{\Yo{1,0}}{\TY}$ for the map induced by the latter. We note that
  $f_{01}$ and $f_{10}$ classify $\IdArr{\Yo{0,1}}$ and $\IdArr{\Yo{1,0}}$,
  respectively.

  Fix $P = \Yo{1,0} \Pushout{\Yo{0,0}} \Yo{0,1}$ and notice that $i^*P$ is the
  coproduct $i^*\Yo{0,1} \Pushout{} i^*\Yo{1,0}$. We therefore amalgamate
  $i^*f_{01}$ and $i^*f_{10}$ into a single morphism:
  \[
    \Mor[f = i^*f_{01} \Pushout{} i^*f_{10}]{i^*P}{i^*\TY}
  \]

  Next, observe that $f$ classifies $\Mor{i^*P}{i^*P}$ and this
  family extends along the monomorphism $\Mor|>->|[m]{i^*P}{\TermObj{}}$ to the
  family $\Mor{\TermObj{}}{\TermObj{}}$. However, there is no morphism
  classifying $\Mor{\TermObj{}}{\TermObj{}}$ that restricts to $f$ along
  $m$. Such a morphism would correspond to a $\VTY$-small presheaf
  $\Mor[G]{\OpCat{\CCat}}{\VTY}$. If such a presheaf were to restrict correctly
  to both $i^*{f_{01}}$ and $i^*{f_{10}}$ correctly, it would need to satisfy
  $G_{00} = \mathbf{2}$ and $G_{00} = \TermObj{}$, which is an impossibility.
\end{proof}

\section{Generalities on descent and \texorpdfstring{$\kappa$}{kappa}-compactness}
\label{sec:descent}

In preparation for our universe construction, we recall notions of descent and compactness together
and develop the required theory. Accordingly, fix a Grothendieck topos
$\ECat$. Unless specifically mentioned otherwise, we shall assume that all
regular cardinals are infinite.

In \cref{sec:intro} we observed that the natural notion of morphism between generic maps $\pi$,
$\rho$ for a universe is not a merely a commuting square $\Mor{\pi}{\rho}$ but rather a
\emph{cartesian} square; only the latter ensures that a family classified by $\pi$ is also
classified by $\rho$. While $\ArrCat{\ECat}$ readily adopts the essential characteristics of $\ECat$
(for instance, it is also a Grothendieck topos) the wide subcategory restricting to cartesian
squares is not even cocomplete. We first recall the descent properties of $\ECat$ to show that this
subcategory is closed under coproducts, filtered colimits and pushouts along monomorphisms
(\cref{lem:colim-in-cart-arr}).

In \cref{sec:hs-and-s} we worked with a universe of presheaves valued in small sets. While
convenient, this definition of smallness relies on a choice of presentation of a topos as a
particular category of presheaves. Under mild restrictions, however, $\PrCls\Sub{\VTY}$ coincides
with the class of relatively \emph{compact} morphisms. Compactness is a `presentation-invariant'
notion and thereby readily available in $\ECat$. We recall the theory of $\kappa$-compactness in
$\ECat$. We show that for sufficiently large $\kappa$, the class of relatively $\kappa$-compact
morphisms form a universe satisfying \textbf{(U1--7)} closed under certain colimits
(\cref{lem:colim-rel-pres,thm:compact-universe}).

\subsection{Descent in a Grothendieck topos}

\begin{definition}\label[definition]{def:descent}
  A diagram $\Mor[J]{\DCat}{\ECat}$ is said to satisfy \emph{descent} when for
  any cartesian natural transformation $\Mor[\alpha]{K}{J}$, the induced
  morphisms $\Mor{\alpha_d}{\Colim{d\in\DCat}{\alpha\Sub{d}}}$ in $\ArrCat{\ECat}$
  are cartesian for each $d\in \DCat$, \ie the following square is cartesian:
  \[
    \DiagramSquare{
      nw/style = pullback,
      width = 2.5cm,
      nw = K\prn{d},
      sw = J\prn{d},
      ne = \Colim{\DCat}{K},
      se = \Colim{\DCat}{J},
    }
  \]
\end{definition}

\begin{remark}
  We caution the reader that the usages of the word \emph{descent} here and in
  \textbf{(U7)} are not identical. A diagram $\Mor[F]{\DCat}{\ECat}$ satisfying
  descent essentially stipulates that we may fully characterize families over
  $\Colim{} F$ by considering cartesian diagrams of families over $F\prn{i}$. In
  particular, all categorical structures from the latter \emph{descend} to
  the former. In contrast, \textbf{(U7)} states that a specific property---that
  of being $\Cls$-small---is amenable to such descent arguments since, in
  particular, pullback along $\Mor{\Coprod{i} F\prn{i}}{\Colim{F}}$ induces a
  suitable cartesian epimorphism.
\end{remark}

We will often speak metonymically of a colimit having descent, to mean that the
diagram of which it is the colimit has descent.

\begin{notation}
  Write $\CartArrCat{\ECat}\subseteq\ArrCat{\ECat}$ for
  the wide subcategory spanned by cartesian maps.
\end{notation}

\begin{lemma}\label[lemma]{lem:colim-in-cart-arr}
  Let $\Mor[J]{\DCat}{\CartArrCat{\ECat}}$ be a diagram whose base
  $\Mor[J_1]{\DCat}{\ECat}$ satisfies descent in the sense of
  \cref{def:descent}. Then the colimit $\Colim{\DCat}{J}$ exists in
  $\CartArrCat{\ECat}$.
\end{lemma}

\begin{proof}
  We may first compute the colimit of $J$ in the ordinary arrow category
  $\ArrCat{\ECat}$. Next we must show that each map
  $\Mor{J\prn{d}}{\Colim{\DCat}{J}}$ is cartesian, but this is exactly the
  content of $J_1$ enjoying descent. We must now check that the factorizations induced by the universal property of this colimit in $\ArrCat{\ECat}$
  are cartesian.

  Fixing a cartesian natural transformation $\Mor[h]{J}{\brc{X}}$, we must check
  that the induced map $\Mor[h^\sharp]{\Colim{\DCat}{J}}{X}$ is cartesian. We
  may cover $\Colim{\DCat}{J}$ by the coproduct $\Coprod{\DCat}{J}$; by the
  descent property of effective epimorphisms, it suffices to check
  that $\Mor|->>|{\Coprod{\DCat}{J}}{\Colim{\DCat}{J}}$ and
  $\Mor{\Coprod{\DCat}{J}}{X}$ are both cartesian.
  To see that $\Mor|->>|{\Coprod{\DCat}{J}}{\Colim{\DCat}{J}}$ is cartesian, it
  suffices to recall that each
  $\Mor{J\prn{d}}{\Colim{\DCat}{J}}$ is cartesian by assumption. Likewise to
  check that $\Mor{\Coprod{\DCat}{J}}{X}$ is cartesian, it suffices to recall
  our assumption that each component $\Mor[h_d]{J\prn{d}}{X}$ is cartesian.
\end{proof}

While all diagrams satisfy descent in an $\infty$-topos, only some diagrams in
$1$-topos theory have descent. The following classes of colimits do enjoy
descent:

\begin{enumerate}

  \item Coproducts enjoy descent: this is one phrasing of the traditional
    disjointness condition that for each $i\not= j$, the fiber product
    $X_i\times\Sub{\Coprod{k}X_k}X_j$ is the initial object:
    \[
      \DiagramSquare{
        nw = \InitObj{\ECat},
        ne = X_i,
        se = \Coprod{k}{X_k},
        sw = X_j,
        nw/style = pullback,
      }
    \]

  \item While pushouts do not generally enjoy descent (see
    \citet[Example~2.3]{rezk:2010} for a counterexample), pushouts along
    monomorphisms do enjoy descent; this property is commonly referred to as
    \emph{adhesivity}~\citep{garner-lack:2012}.

  \item Filtered colimits enjoy descent.

\end{enumerate}

The final condition (verified in \cref{lem:filtered-colimit-descent}) is a
generalization of the \emph{exhaustivity} condition identified by
\citep{shulman:2015:elegant}.

\begin{lemma}\label[lemma]{lem:cosl-proj-final}
  Let $\DCat$ be a filtered category~\citep[Definition
  1.4]{adamek-rosicky:1994}; for any object $d\in\DCat$, the coslice projection
  functor $\Mor[p]{\CoSl{\DCat}{d}}{\DCat}$ is \emph{final}~\citep[Section
  IX.3]{maclane:1998} such that restricting any diagram $\Mor{\DCat}{\ECat}$ to
  the coslice does not change its colimit.
\end{lemma}

\begin{proof}
  We must show that for any object $e\in\DCat$, the comma category $e\downarrow p$
  is connected. Fixing $x,y \in \CoSl{\DCat}{d}$ and $\Mor[i]{e}{p\prn{x}}$ and
  $\Mor[j]{e}{p\prn{y}}$, we must find a zig-zag of morphisms connecting $i$ to
  $j$ in $e\downarrow p$.
  Because $\CoSl{\DCat}{d}$ is filtered, we may find $w\in \CoSl{\DCat}{d}$ with
  $\Mor[m]{x}{w}$ and $\Mor[n]{y}{w}$. We have two triangles that cannot yet be
  pasted into a zig-zag:
  \[
    \begin{tikzpicture}[diagram,baseline=(p/w.base)]
      \node (e) {$e$};
      \node (p/w) [below = of e] {$p\prn{w}$};
      \node (p/x) [left = 2.5cm of p/w] {$p\prn{x}$};
      \node (p/y) [right = 2.5cm of p/w] {$p\prn{y}$};
      \draw[->,bend right=20] (e) to node [sloped,above] {$i$} (p/x);
      \draw[->,bend left=20] (e) to node [sloped,above] {$j$} (p/y);
      \draw[->,bend right=30] (e) to (p/w);
      \draw[->,bend left=30] (e) to (p/w);
      \node [between = e and p/w] {$\puncture$};
      \draw[->] (p/x) to node [below] {$p\prn{m}$} (p/w);
      \draw[->] (p/y) to node [below] {$p\prn{n}$} (p/w);
    \end{tikzpicture}
  \]

  Using the fact that $\DCat$ is filtered, we may find an arrow
  $\Mor{p\prn{w}}{z}$ that unites the two morphisms $\Mor{e}{p\prn{w}}$;
  because $w$ is under $d$ so is $z$, so in fact we have an arrow
  $\Mor[o]{w}{z}$ in $\CoSl{\DCat}{d}$ with which we may complete the connection
  between $i$ and $j$:
  \[
    \begin{tikzpicture}[diagram,baseline=(p/w.base)]
      \node (e) {$e$};
      \node (p/z) [below = of e] {$p\prn{z}$};
      \node (p/w) [below = 2.5cm of p/z] {$p\prn{w}$};
      \node (p/x) [left = 2.5cm of p/z] {$p\prn{x}$};
      \node (p/y) [right = 2.5cm of p/z] {$p\prn{y}$};
      \draw[->,bend right=20] (e) to node [sloped,above] {$i$} (p/x);
      \draw[->,bend left=20] (e) to node [sloped,above] {$j$} (p/y);
      \draw[->] (p/x) to (p/z);
      \draw[->] (p/y) to (p/z);
      \draw[->] (p/x) to node [sloped,below] {$p\prn{m}$} (p/w);
      \draw[->] (p/y) to node [sloped,below] {$p\prn{n}$} (p/w);
      \draw[->] (e) to (p/z);
      \draw[->] (p/w) to node [upright desc] {$p\prn{o}$} (p/z);
    \end{tikzpicture}
    \qedhere
  \]
\end{proof}

\cref{lem:filtered-colimit-descent} below is verified in greater generality by
\citet[Proposition~5.10]{garner-lack:2012:lex-colimits}; we provide a direct proof for
expository purposes.

\begin{lemma}\label[lemma]{lem:filtered-colimit-descent}
  Any filtered diagram $\Mor[F]{\DCat}{\ECat}$ enjoys descent.
\end{lemma}

\begin{proof}
  We fix a cartesian natural transformation $\Mor{G}{F}$ and
  must check for each $d\in\DCat$ the following square is cartesian:
  \[
    \DiagramSquare{
      width = 2.5cm,
      nw/style = dotted pullback,
      ne = \Colim{\DCat}{G},
      se = \Colim{\DCat}{F},
      nw = G\prn{d},
      sw = F\prn{d},
    }
  \]

  Because $\DCat$ is filtered, we may replace the indexing category with the coslice
  $\CoSl{\DCat}{d}$ by \cref{lem:cosl-proj-final}, noting that
  $\Colim{\CoSl{\DCat}{d}}{H} = \Colim{\DCat}{H}$ for any diagram $\Mor[H]{\DCat}{\ECat}$.
  \begin{equation}\label[diagram]{diag:filt-colim-descent:1}
    \DiagramSquare{
      width = 2.5cm,
      nw/style = dotted pullback,
      ne = \Colim{\CoSl{\DCat}{d}}{G},
      se = \Colim{\CoSl{\DCat}{d}}{F},
      nw = G\prn{d},
      sw = F\prn{d},
    }
  \end{equation}

  We observe that any object is the colimit of the constant
  $\CoSl{\DCat}{d}$-diagram it determines as $\CoSl{\DCat}{d}$ is connected;
  therefore we may rewrite \cref{diag:filt-colim-descent:1} as follows:
  \[
    \DiagramSquare{
      width = 3.5cm,
      nw/style = dotted pullback,
      ne = \Colim{\CoSl{\DCat}{d}}{G},
      se = \Colim{\CoSl{\DCat}{d}}{F},
      nw = \Colim{\CoSl{\DCat}{d}}\brc{G\prn{d}},
      sw = \Colim{\CoSl{\DCat}{d}}\brc{F\prn{d}},
    }
  \]

  Recall that filtered colimits commute with finite limits, so it suffices to check that the
  following square below is cartesian for $d \to e$:
  \begin{equation}\label[diagram]{diag:filt-colim-descent:3}
    \DiagramSquare{
      nw/style = dotted pullback,
      nw = G\prn{d},
      ne = G\prn{e},
      sw = F\prn{d},
      se = F\prn{e},
    }
  \end{equation}

  But \cref{diag:filt-colim-descent:3} is cartesian because we have assumed that $\Mor{G}{F}$ is cartesian.
\end{proof}

We recall the notion of \emph{ideal diagram} from \citet{awodey-forssell:2005}.

\begin{definition}
  An \emph{ideal diagram} in a category $\ECat$ is a functor
  $\Mor{\DCat}{\ECat}$ where $\DCat$ is a small filtered preorder and the image
  of each $d\leq e$ is a monomorphism in $\ECat$.
\end{definition}

\begin{lemma}\label[lemma]{lem:filtered-colim-inj-mono}
  If $\Mor[F]{\DCat}{\ECat}$ is an ideal diagram, then each edge $\Mor{F\prn{d}}{\Colim{\DCat}{F}}$ in its colimit cocone is a monomorphism.
\end{lemma}

\begin{proof}
  This follows for essentially the same reason as
  \cref{lem:filtered-colimit-descent}. Fixing $d\in\DCat$, to see that
  $\Mor{F\prn{d}}{\Colim{\DCat}{F}}$ is a monomorphism it suffices to check
  that the following diagram is cartesian:
  \begin{equation}\label[diagram]{diag:colim-inj-mono:0}
    \DiagramSquare{
      nw/style = dotted pullback,
      north/style = {-,double},
      west/style = {-,double},
      ne = F\prn{d},
      se = \Colim{\DCat}{F},
      sw = F\prn{d},
      nw = F\prn{d},
      width = 2.5cm,
    }
  \end{equation}

  Because $\DCat$ is filtered, by \cref{lem:cosl-proj-final} we may
  replace \cref{diag:colim-inj-mono:0} as follows:
  \[
    \DiagramSquare{
      nw/style = dotted pullback,
      north/style = {-,double},
      west/style = {-,double},
      se = \Colim{\CoSl{\DCat}{d}}{F},
      sw = \Colim{\CoSl{\DCat}{d}}{\brc{F\prn{d}}},
      nw = \Colim{\CoSl{\DCat}{d}}{\brc{F\prn{d}}},
      ne = \Colim{\CoSl{\DCat}{d}}{\brc{F\prn{d}}},
      width = 4.25cm,
    }
  \]

  Because filtered colimits commute with finite limits, it suffices to check
  that each of the following squares is cartesian for $e\geq d$:
  \[
    \DiagramSquare{
      nw/style = dotted pullback,
      north/style = {-,double},
      west/style = {-,double},
      se = F\prn{e},
      sw = F\prn{d},
      nw = F\prn{d},
      ne = F\prn{d},
    }
  \]

  But we have already assumed $\Mor{F\prn{d}}{F\prn{e}}$ to be a monomorphism.
\end{proof}

\begin{remark}
  For any regular cardinal $\kappa \ge \omega$, a $\kappa$-filtered
  diagram~\citep[Remark 1.21]{adamek-rosicky:1994} is filtered. Accordingly,
  both \cref{lem:filtered-colimit-descent,lem:filtered-colim-inj-mono} hold for
  $\kappa$-filtered diagrams.
\end{remark}

\begin{lemma}\label[lemma]{lem:good-colim-preserves-monos}
  Let $\Mor[F,G]{\DCat}{\ECat}$ be two diagrams such that $G$ satisfies descent, and
  let $\Mor|>->|{F}{G}$ be a cartesian monomorphism. Then the induced map
  $\Mor{\Colim{\DCat}{F}}{\Colim{\DCat}{G}}$ is a monomorphism.
\end{lemma}

\begin{proof}
  We need to check that the following square is cartesian:
  \[
    \DiagramSquare{
      nw/style = dotted pullback,
      ne = \Colim{\DCat}{F},
      se = \Colim{\DCat}{G},
      sw = \Colim{\DCat}{F},
      nw = \Colim{\DCat}{F},
      width = 3cm,
      north/style = {-,double},
      west/style = {-,double},
    }
  \]

  We can cover $\Colim{\DCat}{F}$ by $\Coprod{\DCat}{F}$; by descent of
  cartesian squares along covers, it suffices to prove that the outer square
  below is cartesian:
  \begin{equation}\label[diagram]{diag:colim-pres-mono:1}
    \begin{tikzpicture}[diagram,baseline=(l/sw.base)]
      \SpliceDiagramSquare<l/>{
        nw/style = pullback,
        ne/style = dotted pullback,
        se = \Colim{\DCat}{F},
        ne = \Colim{\DCat}{F},
        nw = \Coprod{\DCat}{F},
        sw = \Coprod{\DCat}{F},
        east/style = {-,double},
        west/style = {-,double},
        south/style = ->>,
        north/style = ->>,
        width = 2.5cm,
      }
      \SpliceDiagramSquare<r/>{
        glue = west,
        glue target = l/,
        ne = \Colim{\DCat}{F},
        se = \Colim{\DCat}{G},
        width = 3cm,
        north/style = {-,double},
      }
      \node (cp/G) [below = 1.5cm of l/se] {$\Coprod{\DCat}{G}$};
      \draw[>->] (l/sw) to (cp/G);
      \draw[->>] (cp/G) to (r/se);
    \end{tikzpicture}
  \end{equation}

  We have factored the downstairs map of \cref{diag:colim-pres-mono:1} using
  the universal property of the coproduct.
  Our strategy to show that \cref{diag:colim-pres-mono:1} is cartesian is to
  exhibit it as the pasting of two cartesian squares, as hinted by our
  factorization. In particular,
  by pasting pullbacks it is enough to prove that the right-hand square below is cartesian:
  \begin{equation}\label[diagram]{diag:colim-pres-mono:2}
    \begin{tikzpicture}[diagram,baseline=(l/sw.base)]
      \SpliceDiagramSquare<l/>{
        nw/style = pullback,
        ne/style = dotted pullback,
        east/style = >->,
        south/style = >->,
        north/style = {-,double},
        west/style = {-,double},
        nw = \Coprod{\DCat}{F},
        sw = \Coprod{\DCat}{F},
        ne = \Coprod{\DCat}{F},
        se = \Coprod{\DCat}{G},
        width = 2.5cm,
      }
      \SpliceDiagramSquare<r/>{
        glue = west,
        glue target = l/,
        ne = \Colim{\DCat}{F},
        se = \Colim{\DCat}{G},
        width = 2.5cm,
      }
    \end{tikzpicture}
  \end{equation}

  The left-hand square of \cref{diag:colim-pres-mono:2} can be seen to be
  cartesian using our assumption that $\Mor|>->|{F}{G}$ is a monomorphism. To
  see that the right-hand square is cartesian, we will use our descent hypothesis for
  $G$. In particular, it suffices to check that each of the squares
  below is cartesian:
  \[
    \DiagramSquare{
      nw/style = dotted pullback,
      nw = F\prn{d},
      sw = G\prn{d},
      ne = \Colim{\DCat}{F},
      se = \Colim{\DCat}{G},
      width = 2.5cm
    }
  \]

  But this is exactly the condition that $\Mor[G]{\DCat}{\ECat}$ have descent.
\end{proof}

\subsection{Compact objects and relatively compact maps}

We recall some of the theory of compact objects. We refer the reader to \citet{adamek-rosicky:1994}
for a detailed exposition of compact objects and locally presentable categories.

\begin{definition}
  An object $X\in\ECat$ is said to be \emph{$\kappa$-compact} when the
  functor $\Hom[\ECat]{X}{-}$ preserves $\kappa$-filtered colimits.  Following
  \citet{lurie:2009}, a morphism $\Mor{X}{Y}$ is said to be \emph{relatively
  $\kappa$-compact} if for each $\kappa$-compact object $Z$ and
  morphism $\Mor{Z}{Y}$, the pullback $Z\times\Sub{Y} X$ is
  $\kappa$-compact:
  \[
    \DiagramSquare{
      nw/style = pullback,
      nw = Z\times\Sub{Y}X,
      ne = X,
      se = Y,
      sw = Z,
    }
  \]
  More tersely, the fibers of $\Mor{X}{Y}$ over $\kappa$-compact objects are
  $\kappa$-compact.
\end{definition}

\begin{remark}
  We note that the requirement that $\Mor{X}{\TermObj{}}$ be relatively
  $\kappa$-compact is a priori stronger than merely asking $X$ to be
  $\kappa$-compact. Their equivalence amounts to requiring $\kappa$-compact
  objects to be closed under products, which will hold in all cases of
  importance for us.
\end{remark}

\begin{notation}
  We will write $\Cls_\kappa$ for the class of relatively $\kappa$-compact maps in $\ECat$.
\end{notation}

\begin{definition}
  A category $\CCat$ is locally $\kappa$-presentable when $\CCat$ is cocomplete and there is a set
  of $\kappa$-compact objects that generates $\CCat$ under $\kappa$-filtered colimits.
\end{definition}

\begin{notation}
  As a Grothendieck topos, $\ECat$ is locally $\kappa$-presentable for some
  regular cardinal $\kappa$. For the remainder of this subsection, we fix
  $\kappa$ to be such a cardinal.
\end{notation}

The colimit of a diagram in $\ArrCat{\ECat}$ of relatively $\kappa$-compact
morphisms is not necessarily relatively $\kappa$-compact. For
a simple counterexample, consider an object $X$ that is \emph{not}
$\kappa$-compact; then the following pushout of relatively $\kappa$-compact
morphisms is not relatively $\kappa$-compact:
\[
  \begin{tikzpicture}[diagram]
    \SpliceDiagramSquare<back/>{
      width = 3cm,
      height = 2cm,
      nw = \TermObj{},
      sw = X,
      ne = \TermObj{},
      se = \TermObj{},
      east/style = {color=LightGray,->},
      south/style = {color=LightGray,->},
    }
    \SpliceDiagramSquare<front/>{
      width = 3cm,
      height = 2cm,
      nw/style = {below right = 2cm of back/nw},
      north/style = {exists,->},
      south/style = {exists,->},
      nw = X,
      sw = X,
      ne = X,
      se = \TermObj{},
      east/style = {color=RedDevil,->},
    }
    \path[->] (back/sw) edge (front/sw);
    \path[->] (back/nw) edge (front/nw);
    \path[->,exists] (back/ne) edge (front/ne);
    \path[->,exists] (back/se) edge (front/se);
  \end{tikzpicture}
\]

More can be said when the diagram is cartesian (\ie valued in
$\CartArrCat{\ECat}$).  In particular, relatively $\kappa$-compact
morphisms are closed under colimits of cartesian diagrams whose bases satisfy descent in
the sense of \cref{def:descent}, which we verify in
\cref{lem:colim-rel-pres}. We first recall Proposition 4.18 of \citet{shulman:2019}.

\begin{proposition}
  \label[proposition]{prop:colimit-fiberwise-presentability}
  Let $\Mor[J]{\DCat}{\ECat}$ be a diagram and let $Y$ be its colimit; a
  morphism $\Mor{X}{Y}$ is relatively $\kappa$-compact if and only if for
  each $d\in\DCat$, the pullback $\Mor{X\times_Y J\prn{d}}{J\prn{d}}$ depicted
  below is relatively $\kappa$-compact:
  \[
    \DiagramSquare{
      width = 2.5cm,
      nw/style = pullback,
      west/style = RegalBlue,
      ne = X,
      se = Y,
      nw = X\times_Y J\prn{d},
      sw = J\prn{d},
      south = i_d,
    }
  \]
\end{proposition}

\begin{proof}
  The only if direction is clear, so suppose for each $d \in \DCat$,
  $\Mor{X \times_Y J\prn{d}}{J\prn{d}}$ is relatively $\kappa$-compact. We must show that
  $\Mor{X}{Y}$ is relatively $\kappa$-compact.
  Recall that any diagram can be presented as a $\kappa$-filtered diagram of
  colimits of $\kappa$-small sub-diagrams~\citep[Theorem
  IX.1.1]{maclane:1998}. Therefore, it suffices to show that this holds when $J$
  is $\kappa$-filtered and when $J$ is $\kappa$-small.

  First suppose $J$ is $\kappa$-filtered. Fix a $\kappa$-compact object $Z$ together with a morphism
  $\Mor{Z}{Y}$, we must show that the pullback $Z \times_Y X$ is $\kappa$-compact. As $Y$ is the
  colimit of a $\kappa$-filtered diagram, the morphism $\Mor{Z}{Y}$ must factor through some
  $\Mor{J\prn{d}}{Y}$:
  \[
    \begin{tikzpicture}[diagram]
      \SpliceDiagramSquare<l/>{
        width = 2.5cm,
        nw/style = pullback,
        ne/style = pullback,
        nw = {Z \times_Y X},
        sw = Z,
        ne = {J\prn{d} \times_Y X},
        se = J\prn{d},
      }
      \SpliceDiagramSquare<r/>{
        glue = west, glue target = l/,
        ne = X,
        se = Y,
      }
    \end{tikzpicture}
  \]
  By assumption, $\Mor{J\prn{d} \times_Y X}{J\prn{d}}$ is relatively $\kappa$-compact so
  $Z \times_Y X$ is $\kappa$-compact.

  Next, suppose that $J$ is a $\kappa$-small diagram. In this case, the diagram category
  $\ECat^\DCat$ is also locally $\kappa$-presentable~\citep[Corollary
  1.54]{adamek-rosicky:1994}. Accordingly, $D = \Colim{i \in \ICat} E_i$, where each $E_i$ is a
  $\kappa$-compact object in $\ECat^\DCat$ and $\ICat$ is $\kappa$-filtered. Each $E_{i}(d)$ is
  $\kappa$-compact~\citep[Lemma 4.2]{shulman:2019} and by commutation of colimits
  $Y = \Colim{i \in \ICat} \Colim{d \in \DCat} E_i(d)$.

  By assumption $\ICat$ is $\kappa$-filtered so by the already proven case it suffices to show that
  $\Mor{X \times_{Y} \Colim{d} E_i\prn{d}}{\Colim{d} E_i\prn{d}}$ is relatively $\kappa$-compact for
  each $i \in \ICat$. As the $\kappa$-small colimit of $\kappa$-small objects, $\Colim{d} E_i\prn{d}$
  is $\kappa$-compact so this morphism is relatively $\kappa$-compact if and only if
  $X \times_{Y} \Colim{d} E_i(d)$ is $\kappa$-compact. By universality of colimits, we have a
  sequence of identifications:
  \[
    X \times_Y \Colim{d} E_i\prn{d} =
    \Colim{d} X \times_Y E_i\prn{d} =
    \Colim{d} \prn{\prn{X \times_{Y} J\prn{d}} \times_{J(d)} E_i\prn{d}}
  \]

  Thus, this object is $\kappa$-compact as the $\kappa$-small colimit of $\kappa$-compact objects.
\end{proof}

\begin{lemma}\label[lemma]{lem:colim-rel-pres}
  The colimit of a diagram $\Mor[J]{\DCat}{\CartArrCat{\ECat}}$ of relatively
  $\kappa$-compact morphisms is relatively $\kappa$-compact if the base
  $\Mor[J_1]{\DCat}{\ECat}$ satisfies descent in the sense of \cref{def:descent}.
\end{lemma}

\begin{proof}
  By \cref{prop:colimit-fiberwise-presentability} it suffices to check that
  each fiber $i_d^*\Colim{\DCat}{J_0} : \ArrCat{\ECat}$ below is relatively
  $\kappa$-compact:
  \begin{equation}\label[diagram]{diag:colim-rel-pres}
    \DiagramSquare{
      width = 3.5cm,
      nw/style = pullback,
      ne = J\prn{d},
      ne = \Colim{\DCat}J_0,
      se = \Colim{\DCat}J_1,
      sw = J_1\prn{d},
      south = i_d,
      nw = i_d^*\Colim{\DCat}{J_0},
    }
  \end{equation}

  Because $J_1$ satisfies descent, the cartesian square
  depicted in \cref{diag:colim-rel-pres} is actually
  $\Mor{J\prn{d}}{\Colim{\DCat}{J}}$; but we have already assumed that $J\prn{d}$ is
  relatively $\kappa$-compact.
\end{proof}

\begin{lemma}
  \label[lemma]{lem:cls-kappa-satisfies-descent}
  The class of maps $\Cls_\kappa$ satisfies the descent axiom \textbf{(U7)}.
\end{lemma}

\begin{proof}
  Let $g$ be a relatively $\kappa$-compact morphism equipped
  with a cartesian epimorphism $\Mor|->>|{g}{f}$ as below:
  \[
    \DiagramSquare{
      nw/style = pullback,
      south/style = ->>,
      north/style = ->>,
      nw = C,
      sw = D,
      ne = A,
      se = B,
      west = \Cls_\kappa\ni g,
      east = f,
      south = b,
      north = a,
    }
  \]

  We must show that $f$ is relatively $\kappa$-compact. We will use the fact
  that both $\Mor|->>|[a]{C}{A}$ and $\Mor|->>|[b]{D}{B}$ are coequalizers of
  their kernel pairs, and that kernel pairs are stable:
  \begin{equation}\label[diagram]{diag:kappa-pres-u7:1}
    \begin{tikzpicture}[diagram,baseline=(0.base)]
      \node (A) [pullback] {$C\times_AC$};
      \node (B) [pullback, right = 2.5cm of A] {$C$};
      \node (C) [right = 2.5cm of B] {$A$};
      \path[->>] (B) edge node [above] {$a$} (C);
      \path[->,transform canvas={yshift=.1cm}] (A) edge node [above] {$q_1$} (B);
      \path[->,transform canvas={yshift=-.1cm}] (A) edge node [below] {$q_2$} (B);
      \node (0) [below = 2cm of A] {$D\times_BD$};
      \node (1) [right = 2.5cm of 0] {$D$};
      \node (2) [right = 2.5cm of 1] {$B$};
      \path[->>] (1) edge node [above] {$b$} (2);
      \path[->,transform canvas={yshift=.1cm}] (0) edge node [above] {$p_1$} (1);
      \path[->,transform canvas={yshift=-.1cm}] (0) edge node [below] {$p_2$} (1);
      \path[->] (A) edge (0);
      \path[->] (B) edge node[upright desc] {$g$} (1);
      \path[->] (C) edge node[right] {$f$} (2);
    \end{tikzpicture}
  \end{equation}

  By \cref{prop:colimit-fiberwise-presentability} it suffices to check that
  $b^*f$, $\prn{b\circ p_0}^*f$, and $\prn{b\circ p_1}^*f$ are relatively
  $\kappa$-compact.  But each of these is a pullback of $g$
  (\cref{diag:kappa-pres-u7:1}) and therefore by stability \textbf{(U1)}, $f$ is
  relatively $\kappa$-compact.
\end{proof}

\subsection{Relating small and relatively compact maps}

For this subsection, fix a presentation $\ECat = \Sh{\CCat,J}$ and
write $i^* \dashv i_*$ for the geometric embedding
$\EmbMor{\Sh{\CCat,J}}{\Psh{\CCat}}$. Recall that a presheaf $P \in \Psh{\CCat}$
is $\kappa$-small when each $P(C)$ is a $\kappa$-small set. Under mild
assumptions, small presheaves precisely correspond to compact presheaves. We
reproduce a proof due to \citet[Example 1.31]{adamek-rosicky:1994}:

\begin{lemma}
  \label[lemma]{lem:small-iff-compact}
  Given a regular cardinal $\kappa > \verts{\CCat}$ and a presheaf $P \in \Psh{\CCat}$, the latter is $\kappa$-compact if and
  only if it is valued in $\kappa$-small sets.
\end{lemma}
\begin{proof}
  First express $P$ as the colimit of representables:
  $P = \Colim{\prn{c, p} \in \ElCat{P}} \Yo{c} = \Colim{\ElCat{P}} \Yo \circ \pi$.
  On one hand, if $P$ is valued in $\kappa$-small sets, then
  $\ElCat{P}$ is $\kappa$-small, while each $\Yo{c}$ is
  $\kappa$-compact.  Thus, $P$ is a $\kappa$-small colimit of
  $\kappa$-compact objects, hence $\kappa$-compact.

  On the other hand, suppose instead that $P$ is $\kappa$-compact;
  we will show that it is valued in $\kappa$-small sets.
  By completing $\ElCat{P}$ under $\kappa$-small colimits and extending
  $\Yo \circ \pi$ by colimits, we obtain a $\kappa$-filtered diagram $\DCat$ and
  a map $\Mor[F]{\DCat}{\Psh{\CCat}}$ which sends a formal colimit to a
  $\kappa$-small colimit of representables. Observe that each $F(d)$ is
  $\kappa$-small as a $\kappa$-small colimit of representables.
  Moreover, the canonical map $\Mor[p]{\Colim{\DCat} F}{P}$ is an
  isomorphism~\citep[Theorem 1.20]{adamek-rosicky:1994} so that, in particular,
  $P$ is the $\kappa$-filtered colimit of $\kappa$-small objects.

  As $P$ is $\kappa$-compact, we obtain a map $\Mor[r]{P}{F\prn{d}}$ for some
  $d : \DCat$ fitting into the following diagram:
  \[
    \begin{tikzpicture}[diagram]
      \node (P) {$P$};
      \node [below right = 1.5cm and 3cm of P] (Fd) {$F\prn{d}$};
      \node [right = 3cm of P] (C) {$\Colim{\DCat} F$};
      \path[->] (P) edge node[above] {$p^{-1}$} (C);
      \path[->] (P) edge node[sloped,below] {$r$} (Fd);
      \path[->] (Fd) edge (C);
    \end{tikzpicture}
  \]
  It follows immediately that $r$ is monic, and so $P$ is a subobject of
  $F\prn{d}$. As $F\prn{d}$ is valued in $\kappa$-small sets, so is $P$.
\end{proof}

\begin{lemma}
  \label[lemma]{lem:rel-compact-iff-repr-compact}
  For any $\kappa > \verts{\CCat}$, a morphism $\Mor[f]{P}{Q}$ is relatively $\kappa$-compact in
  $\Psh{\CCat}$ if and only if the fibers of $f$ over representable presheaves are
  $\kappa$-compact.
\end{lemma}
\begin{proof}
  The only-if direction is immediate, so it suffices to show that $f$ is relatively compact provided
  that its fibers over representable presheaves are compact. To this end, fix a $\kappa$-compact
  presheaf $R$ and a morphism $\Mor[g]{R}{Q}$:
  \[
    \DiagramSquare{
      nw = g^*{P},
      nw/style = pullback,
      ne = P,
      se = Q,
      sw = R,
      south = g,
      east = f,
    }
  \]

  We must show that $g^*{P}$ is $\kappa$-compact. Viewing $R$ as a colimit of
  representables, universality ensures that $g^*{P} = \Colim{\prn{C,r} \in
  \ElCat{R}} f^*\Yo{C}$. By assumption, each $f^*\Yo{C}$ is $\kappa$-compact,
  and by \cref{lem:small-iff-compact} $\ElCat{R}$ is a $\kappa$-small category.
  Accordingly, as a $\kappa$-small colimit of $\kappa$-compact objects, $g^*P$
  is $\kappa$-compact.
\end{proof}

For the next sequence of results, we shall require some results from the theory
of accessible categories and accessible functors. In order to state them, we
require a small amount of set-theoretic bureaucracy in the form of the
$\SharplyGt$ relation:

\begin{definition}
  A cardinal $\lambda > \kappa$ is \emph{sharply larger} than $\kappa$, notated
  $\lambda \SharplyGt \kappa$, if each $\kappa$-accessible category is
  $\lambda$-accessible.
\end{definition}

We emphasize that $\lambda \SharplyGt \kappa$ is not the same as
$\lambda > \kappa$ nor does it mean anything akin to ``$\lambda$ is much larger
than $\kappa$''. We refer the reader to \textcite[Theorem
2.11]{adamek-rosicky:1994} for more information about $\SharplyGt$. For our
purposes it suffices to know that if $\lambda$ is strongly inaccessible then
$\kappa < \lambda$ is equivalent to $\kappa \SharplyLt \lambda$.

\begin{lemma}
  \label[lemma]{lem:incl-preserves-compactness}
  There exists a cardinal $\lambda_0$ such that for any
  $\lambda \SharplyGt \lambda_0$, both $i_*$ and $i^*$ preserve
  $\lambda$-filtered colimits and $\lambda$-compact objects.
\end{lemma}
\begin{proof}
  As adjoints $i_*$ and $i^*$ are both accessible functors. Therefore, the result follows immediately
  from the uniformization result (2.19) of \citet{adamek-rosicky:1994}.
\end{proof}

\begin{lemma}
  \label[lemma]{lem:incl-reflects-compactness}
  If $i^*$ preserves $\lambda$-compact objects, then $i_*$ reflects them.
\end{lemma}

\begin{proof}
  Let $E\in\Sh{\CCat,J}$ be such that $i_*E$ is $\lambda$-compact; because $i^*$
  preserves $\lambda$-compact objects, $i^*i_*E \cong E$ is $\lambda$-compact.
\end{proof}

Combining the above result with the characterization of $\kappa$-compact objects
given by \cref{lem:small-iff-compact}, we deduce the following.

\begin{corollary}
  \label[corollary]{cor:sharply-larger-facts}
  Given a regular cardinal $\lambda$ sharply larger than both
  $\lambda_0$ and $\verts{\CCat}$, the following properties hold:
  \begin{enumerate}
  \item $\ECat$ is locally $\lambda$-presentable.
  \item The $\lambda$-compact objects in $\ECat$ are closed under finite limits.
  \end{enumerate}

  If $\lambda$ is further assumed to be strongly inaccessible, then we
  additionally have:
  \begin{enumerate}[resume]
  \item The set $\Hom[\ECat]{X}{Y}$ between two $\lambda$-compact objects $X,Y$ is $\lambda$-small.
  \end{enumerate}
\end{corollary}

\begin{lemma}
  \label[lemma]{lem:dir-image-pres-and-reflects}
  Given a regular cardinal $\lambda$ sharply larger than both $\lambda_0$ and $\verts{\CCat}$,
  the direct image functor $i_*$ preserves and reflects
  relatively $\lambda$-compact morphisms.
\end{lemma}

\begin{proof}
  We handle preservation and reflection separately.

  \emph{Preservation.}
  Let $\Mor{X}{Y}$ be a relatively $\lambda$-compact morphism in $\ECat$. We
  must check that $\Mor{i_*X}{i_*Y}$ is relatively $\lambda$-compact.
  Fixing a $\lambda$-compact object $Z \in \Psh{\CCat}$ along with a map
  $\Mor{Z}{i_*Y}$, it suffices to argue that the fiber product $W = Z
  \times_{i_*Y} i_*X$ is $\lambda$-compact:
  \[
    \DiagramSquare{
      nw = W,
      nw/style = pullback,
      sw = Z,
      ne = i_*X,
      se = i_*Y,
    }
  \]

  Observe that $\Mor{Z}{i_*Y}$ factors uniquely through
  $\Mor[\eta_Z]{Z}{i_*i^*Z}$. As $i_*$ preserves cartesian squares, we can
  factor the above cartesian square as follows:
  \[
    \begin{tikzpicture}[diagram]
      \SpliceDiagramSquare<l/>{
        width = 2.75cm,
        nw = W,
        nw/style = pullback,
        sw = Z,
        ne = i_*\prn{i^* Z \times_Y X},
        ne/style = pullback,
        se = i_*i^*Z,
      }
      \SpliceDiagramSquare<r/>{
        width = 2.75cm,
        glue = west,
        glue target = l/,
        ne = i_*X,
        se = i_*Y,
      }
    \end{tikzpicture}
  \]

  Recalling that $i^*$ preserves $\lambda$-compact objects
  (\cref{lem:incl-preserves-compactness}), $i^*Z$ is $\lambda$-compact and
  consequently so too is $i^*Z \times_{Y} X$. By
  \cref{lem:incl-preserves-compactness} again, both $i_*i^*Z$ and
  $i_*\prn{i^*Z \times_{Y} X}$ are $\lambda$-compact. Finally, $W$ is
  $\lambda$-compact as the finite limit of $\lambda$-compact objects
  (\cref{cor:sharply-larger-facts}).

  \emph{Reflection.}
  Let $\Mor{X}{Y}$ be a morphism in $\ECat$ such $\Mor{i_*X}{i_*Y}$ is
  relatively $\lambda$-compact in $\Psh{\CCat}$. Fixing a morphism
  $\Mor{Z}{Y}$ with $Z$ a $\lambda$-compact object, we must check that the
  fiber product $W$ below is $\lambda$-compact:
  \[
    \DiagramSquare{
      nw/style = pullback,
      north/style = {exists,->},
      west/style = {exists,->},
      nw = W,
      sw = Z,
      se = Y,
      ne = X,
    }
  \]

  The right adjoint $i_*$ preserves $\lambda$-compact objects by assumption
  (\cref{lem:incl-preserves-compactness}) and hence $i_*Z$ is
  $\lambda$-compact; because $i_*$ also preserves pullbacks, we deduce that
  $i_*W$ is a $\lambda$-compact object in $\Psh{\CCat}$:
  \[
    \DiagramSquare{
      nw/style = pullback,
      nw = i_*W,
      sw = i_*Z,
      se = i_*Y,
      ne = i_*X,
    }
  \]

  Finally, \cref{lem:incl-reflects-compactness} implies $W$ is $\lambda$-compact.
\end{proof}

\begin{remark}
  The proof of \cref{lem:dir-image-pres-and-reflects} establishes a more general
  result: a right adjoint $\Mor[G]{\CCat}{\DCat}$ between finitely complete
  categories preserves relatively $\kappa$-compact families provided both
  adjoints preserve $\kappa$-compact objects and $\kappa$-compact objects in
  $\DCat$ are closed under finite limits. If $G$ is additionally assumed to
  reflect $\kappa$-compact objects, it reflects $\kappa$-compact families.
\end{remark}

Combining the above results with \cref{thm:s}, we obtain the following result:

\begin{theorem}
  \label[theorem]{thm:compact-universe}
  There exists a cardinal $\kappa$ such that for any strongly inaccessible $\lambda \SharplyGt \kappa$, $\ECat$ is locally
  $\lambda$-presentable and the class of relatively $\lambda$-compact maps in $\ECat$ form a
  universe $\Cls_\lambda$ satisfying \textbf{(U1--7)} and $\lambda$-compact objects are closed under
  finite limits.
\end{theorem}
\begin{proof}
  We define $\kappa$ to be any regular cardinal sharply larger than both $\lambda_0$ and $\verts{\CCat}$. We first recall that $\ECat$
  is locally $\lambda$-presentable and that $\lambda$-compact objects are closed under finite limits
  by \cref{cor:sharply-larger-facts}. Next, \cref{thm:s} combined with
  \cref{lem:dir-image-pres-and-reflects,lem:rel-compact-iff-repr-compact,lem:small-iff-compact}
  ensures that for any $\lambda \SharplyGt \kappa$, the universe $\Cls_\lambda$ satisfies
  \textbf{(U1--6)}. Finally, we have established that $\Cls_\lambda$ satisfies \textbf{(U7)} in
  \cref{lem:cls-kappa-satisfies-descent}.
\end{proof}

\begin{definition}
  We write $\Threshold\prn{\ECat}$ for the cardinal $\kappa$ provided by
  \cref{thm:compact-universe}.
\end{definition}

\begin{corollary}\label[corollary]{lem:pres-sl-ess-small}
  For any strongly inaccessible $\lambda \SharplyGt \Threshold\prn{\ECat}$, the
  full subcategory of $\Sl{\ECat}{Y}$ spanned by relatively $\lambda$-compact
  maps is essentially small.
\end{corollary}

\begin{proof}
  Writing $\Mor[\El_\lambda]{\EL_\lambda}{\TY_\lambda}$ for the generic map of $\Cls_\lambda$, this
  subcategory of $\Sl{\ECat}{Y}$ is bounded by $\Hom{Y}{\TY_\lambda}$.
\end{proof}

\begin{lemma}
  \label[lemma]{lem:small-set-of-monos}
  For any strongly inaccessible $\lambda \SharplyGt \Threshold\prn{\ECat}$,
  there exists a $\lambda$-small set of monomorphisms $\GenCls$ generating all
  monomorphisms in $\ECat$ under pushout, transfinite composition, and retracts.
  Moreover, the domains and codomains of morphisms in $\GenCls$ are $\lambda$-compact.
\end{lemma}
\begin{proof}
  \citet[Proposition 1.12]{beke:2000} shows that the collection of sub-quotients of representables
  $J$ generate all monomorphisms in $\Psh{\CCat}$. Explicitly, $J$ is the collection of
  monomorphisms $\Mor|>->|{A}{B}$ where $B$ is the quotient of a representable $\Yo{C}$. As
  $\Psh{\CCat}$ is both well-powered and co-well-powered there is essentially a set of such
  monomorphisms.

  A quotient of a representable $\Yo{C}$ is determined by a morphism
  $\Mor{\Yo{C} \times \Yo{C}}{\Omega}$. As $\lambda > \verts{\CCat}$, $\Omega$ is $\lambda$-small
  and there is a $\lambda$-small set of representables therefore $J$ may be chosen to be
  $\lambda$-small. Finally, the domains and codomains of monomorphisms in $J$ are $\lambda$-small,
  since they are subquotients of representables which are $\lambda$-small;
  and by \cref{lem:small-iff-compact}, this implies they are $\lambda$-compact.

  We now define $\GenCls \subset \Hom[\ECat]$ as the image of $J$ under $i^*$. As $i_*$ preserves
  monomorphisms and $i^*$ preserves all colimits, $\GenCls$ generates all monomorphisms in $\ECat$
  under pushout, transfinite composition, and retracts. The domains and codomains of morphisms in
  $\GenCls$ are seen to be $\lambda$-compact by \cref{lem:incl-preserves-compactness}.
\end{proof}

\section{Main result: a universe satisfying realignment}\label{sec:universe}

Let $\ECat$ be a Grothendieck topos and fix a strongly inaccessible cardinal
$\kappa \SharplyGt \Threshold\prn{\ECat}$.
We have previously shown that $\Cls_\kappa$ satisfies \textbf{(U1--7)}. We construct a new generic
map for this class and thereby conclude that $\Cls_\kappa$ satisfies \textbf{(U8)}.

\subsection{Saturation of solvable realignment problems}\label{sec:saturation}

In \cref{def:realignment} we specified what it means for a universe to have
realignment for a class of monomorphisms $\mathcal{M}$. On the other
hand, any pullback-stable class of maps $\Cls$ and morphism $\Mor[\pi]{E}{U}\in \Cls$
determines a class $\GlCls$ of monomorphisms along which realignment problems can
be solved (regardless of whether $\Cls$ is a universe and whether $\pi$ is
generic).

\begin{notation}\label[notation]{nota:gl-cls}
  We will write $\GlCls$ for the set of all monomorphisms in $\ECat$
  with respect to which $\prn{\Cls,\pi}$ satisfies the realignment property.
\end{notation}

We will establish the closure of $\GlCls$ under pushout,
transfinite composition, and retracts.

\begin{lemma}
  The class of realignable monomorphisms $\GlCls$ is stable under pushout.
\end{lemma}

\begin{proof}
  Fix $\Mor|>->|{A}{B}\in\GlCls$ and a pushout diagram in the following configuration:
  \begin{equation}\label[diagram]{diag:gl-po-st:0}
    \DiagramSquare{
      nw = A,
      sw = B,
      ne = C,
      se = D,
      west/style = >->,
      east/style = >->,
      se/style = pushout,
    }
  \end{equation}

  We must show that $\Mor|>->|{C}{D}\in\GlCls$; to that end, we fix a realignment
  problem $\GlSpan{f}{h}{\pi}$ whose extent lies over $\Mor|>->|{C}{D}$.
  \begin{equation}\label[diagram]{diag:gl-po-st:1}
    \begin{tikzpicture}[diagram,baseline=(sq/sw.base)]
      \SpliceDiagramSquare<sq/>{
        width = 3.5cm,
        nw = h,
        sw = C,
        ne = \pi,
        se = U,
        west/style = lies over,
        east/style = lies over,
        south/style = {color = {LightGray}}
      }
      \node (n) [between = sq/nw and sq/ne, yshift = -0.75cm] {$f$};
      \node (s) [between = sq/sw and sq/se, yshift = -0.75cm] {$D$};
      \node (A) [left = of sq/sw] {$A$};
      \node (B) [left = of s] {$B$};
      \path[->] (A) edge (sq/sw);
      \path[->] (B) edge (s);
      \path[>->] (A) edge (B);
      \path[lies over] (n) edge (s);
      \path[>->] (sq/nw) edge (n);
      \path[>->] (sq/sw) edge (s);

      \node (Arr/E) [right = of sq/ne] {$\ArrCat{\ECat}$};
      \node (E) [right = of sq/se] {$\ECat$};
      \path[fibration] (Arr/E) edge (E);
    \end{tikzpicture}
  \end{equation}

  We will transform the realignment problem of \cref{diag:gl-po-st:1} into one that we
  can already solve; first we fill in the cartesian lifts over the pushout
  square in the base.
  \begin{equation}\label[diagram]{diag:gl-po-st:2}
    \begin{tikzpicture}[diagram,baseline=(sq/sw.base)]
      \SpliceDiagramSquare<sq/>{
        width = 3.5cm,
        nw = h,
        sw = C,
        ne = \pi,
        se = U,
        east/style = lies over,
        south/style = {color = {LightGray}},
        west/style = {lies over,color = {LightGray}}
      }
      \node (f) [between = sq/nw and sq/ne, yshift = -0.75cm] {$f$};
      \node (D) [between = sq/sw and sq/se, yshift = -0.75cm] {$D$};
      \node (A) [left = 3.5cm of sq/sw] {$A$};
      \node (B) [left = 3.5cm of D] {$B$};
      \node (A*h) [left = 3.5cm of sq/nw] {$A^*h$};
      \node (B*f) [left = 3.5cm of f] {$B^*f$};
      \path[->,exists,color=magenta] (A*h) edge (sq/nw);
      \path[lies over] (A*h) edge (A);
      \path[>->,exists,color=magenta] (A*h) edge (B*f);
      \path[->,exists,color=magenta] (B*f) edge (f);
      \path[lies over] (B*f) edge (B);

      \path[->,color=LightGray] (A) edge (sq/sw);
      \path[->] (B) edge (D);
      \path[>->] (A) edge (B);
      \path[lies over] (f) edge (D);
      \path[>->,color=magenta] (sq/nw) edge (f);
      \path[>->] (sq/sw) edge (D);

      \node (Arr/E) [right = of sq/ne] {$\ArrCat{\ECat}$};
      \node (E) [right = of sq/se] {$\ECat$};
      \path[fibration] (Arr/E) edge (E);
    \end{tikzpicture}
  \end{equation}

  By the universality of colimits, the \textcolor{magenta}{upper face} is a
  pushout; therefore to solve our realignment problem, it suffices to find a
  map $\Mor{B^*f}{\pi}$ making the following square commute:
  \begin{equation}\label[diagram]{diag:gl-po-st:3}
    \DiagramSquare{
      nw = A^*h,
      sw = B^*f,
      ne = h,
      se = \pi,
      west/style = {>->,magenta},
      north/style = {->,magenta},
      south/style = {->,exists},
    }
  \end{equation}

  Because $\Cls$ is stable under pullback, we have $B^*f\in\Cls$; therefore
  \cref{diag:gl-po-st:3} is itself a realignment problem whose extent lies over an
  element of $\GlCls$.
\end{proof}

\begin{notation}
  We will write $\OrdToCat{\alpha}$ for the filtered poset of ordinal numbers
  $\beta < \alpha$.
\end{notation}

\begin{lemma}
  The class of realignable monomorphisms $\GlCls$ is stable under transfinite composition.
\end{lemma}

\begin{proof}
  Let $\Mor[F]{\OrdToCat{\alpha}}{\ECat}$ be a cocontinuous functor such that each
  $\Mor{F\prn{\beta}}{F\prn{\beta+1}}$ is an element of $\GlCls$. We must show that the transfinite
  composition $\Mor{F\prn{0}}{\Colim{\OrdToCat{\alpha}}{F}}$ is an element of $\GlCls$.
  We fix a realignment situation
  whose extent lies over some $\Mor|>->|{F\prn{0}}{\Colim{\OrdToCat{\alpha}}{F}}$:
  \begin{equation}\label[diagram]{diag:gl-trans-comp-st:0}
    \begin{tikzpicture}[diagram,baseline=(sq/sw.base)]
      \SpliceDiagramSquare<sq/>{
        width = 3.75cm,
        nw = f_0,
        sw = F\prn{0},
        ne = \pi,
        se = U,
        west/style = lies over,
        east/style = lies over,
        south/style = {color = {LightGray}}
      }
      \node (f) [between = sq/nw and sq/ne, yshift = -1cm] {$f$};
      \node (F/w) [between = sq/sw and sq/se, yshift = -1cm] {$\Colim{\OrdToCat{\alpha}}{F}$};
      \path[lies over] (f) edge (F/w);
      \path[>->] (sq/nw) edge (f);
      \path[>->] (sq/sw) edge (F/w);

      \node (Arr/E) [right = of sq/ne] {$\ArrCat{\ECat}$};
      \node (E) [right = of sq/se] {$\ECat$};
      \path[fibration] (Arr/E) edge (E);
    \end{tikzpicture}
  \end{equation}

  By the universality of colimits, we may replace
  $f$ with $\Colim{\OrdToCat{\alpha}}{f_\bullet}$ where we define
  $\Mor[f_\bullet]{\OrdToCat{\alpha}}{\CartArrCat{\ECat}}$ to extend $f_0$ by
  sending each $f_\beta$ to the following cartesian lift:
  \begin{equation}\label[diagram]{diag:gl-trans-comp-st:1}
    \begin{tikzpicture}[diagram,baseline=(sq/sw.base)]
      \SpliceDiagramSquare<sq/>{
        width = 2.5cm,
        nw/style = pullback,
        west/style = lies over,
        east/style = lies over,
        north/style = {exists,>->},
        south/style = >->,
        nw = f_\beta,
        sw = F\prn{\beta},
        ne = f,
        se = \Colim{\OrdToCat{\alpha}}{F}
      }
      \node (Arr/E) [right = of sq/ne] {$\ArrCat{\ECat}$};
      \node (E) [right = of sq/se] {$\ECat$};
      \path[fibration] (Arr/E) edge (E);
    \end{tikzpicture}
  \end{equation}

  Our realignment problem can therefore be rewritten as follows:
  \begin{equation}\label[diagram]{diag:gl-trans-comp-st:2}
    \begin{tikzpicture}[diagram,baseline=(sw.base)]
      \node (f/0) {$f_0$};
      \node (f) [below = of f/0] {$\Colim{\OrdToCat{\alpha}}{f_\bullet}$};
      \node (pi) [right = of f/0] {$\pi$};
      \draw[>->] (f/0) to (f);
      \draw[->] (f/0) to (pi);
    \end{tikzpicture}
  \end{equation}

  We will define the natural transformation
  $\Mor{\Colim{\OrdToCat{\alpha}}{f_\bullet}}{\pi}$ by transfinite induction on
  $\beta\leq \alpha$.  In the zero case, we use our existing map
  $\Mor{f_0}{\pi}$.  In the successor case, we assume a map
  $\Mor{f_\beta}{\pi}$ extending $\Mor|>->|{f_0}{f_\beta}$ and glue
  $\Mor{f_\beta}{\pi}$ along $\Mor|>->|{f_\beta}{f_{\beta+1}}\in\GlCls$:
  \begin{equation}\label[diagram]{diag:gl-trans-comp-st:3}
    \begin{tikzpicture}[diagram,baseline=(sw.base)]
      \node (f/0) {$f_0$};
      \node (f/beta) [below = of f/0] {$f_\beta$};
      \node (f/beta+1) [below = of f/beta] {$f_{\beta+1}$};
      \node (pi) [right = of f/0] {$\pi$};
      \draw[>->] (f/0) to (f/beta);
      \draw[>->] (f/beta) to (f/beta+1);
      \draw[->] (f/0) to (pi);
      \draw[->] (f/beta) to (pi);
      \draw[->,bend right=20,exists] (f/beta+1) to (pi);
    \end{tikzpicture}
  \end{equation}

  The limit case is trivial, as we may assemble all the prior solutions into a
  single one using the universal property of $f_\beta$ as
  $\Colim{\OrdToCat{\beta}}{f_\bullet}$:
  \begin{equation}\label[diagram]{diag:gl-trans-comp-st:4}
    \begin{tikzpicture}[diagram,baseline=(sw.base)]
      \node (f/0) {$f_0$};
      \node (f) [below = of f/0] {$\Colim{\OrdToCat{\beta}}{f_\bullet} = f_\beta$};
      \node (pi) [right = 3cm of f,yshift = -1cm] {$\pi$};
      \draw[>->] (f/0) to (f);
      \draw[->,bend left=30] (f/0) to (pi);
      \draw[->,exists] (f) to (pi);
      \node (f/alpha) [color=LightGray,left = 3cm of f,yshift=-1cm] {$f_{\alpha<\beta}$};
      \draw[>->,color=LightGray] (f/alpha) to (f);
      \draw[>->,bend right=30,color=LightGray] (f/0) to (f/alpha);
      \draw[->,bend right=30,color=LightGray] (f/alpha) to (pi);
    \end{tikzpicture}
  \end{equation}
  We note that $\Mor{\Colim{\OrdToCat{\beta}}{f_\bullet}}{\pi}$ remains
  cartesian by the universality of colimits.

  The extension to \cref{diag:gl-trans-comp-st:4} remains natural because
  we have merely combined the solutions to the smaller realignment problems.
\end{proof}

\begin{lemma}
  The class of realignable monomorphisms $\GlCls$ is closed under retracts.
\end{lemma}

\begin{proof}
  We fix $\Mor|>->|[j]{A}{B}\in\GlCls$ and a retract $\Mor|>->|[i]{C}{D}$ of
  $j$ in $\ArrCat{\ECat}$:
  \begin{equation}\label[diagram]{diag:gl-ret-st:0}
    \begin{tikzpicture}[diagram,baseline=(l/sw.base)]
      \SpliceDiagramSquare<l/>{
        nw = C,
        sw = D,
        ne = A,
        se = B,
        west = i,
        east = j,
        west/style = >->,
        east/style = >->,
        east/node/style = upright desc,
      }
      \SpliceDiagramSquare<r/>{
        glue = west, glue target = l/,
        ne = C,
        se = D,
        east = i,
        east/style = >->,
      }
      \draw[double,bend left=30] (l/nw) to (r/ne);
      \draw[double,bend right=30] (l/sw) to (r/se);
    \end{tikzpicture}
  \end{equation}

  To check that $i\in\GlCls$, we fix a realignment problem whose
  extent lies over $\Mor|>->|[i]{C}{D}$:
  \begin{equation}\label[diagram]{diag:gl-ret-st:1}
    \begin{tikzpicture}[diagram,baseline=(sq/sw.base)]
      \SpliceDiagramSquare<sq/>{
        width = 3.5cm,
        nw = h,
        sw = C,
        ne = \pi,
        se = U,
        west/style = lies over,
        east/style = lies over,
        south/style = {color = {LightGray}}
      }
      \node (f) [between = sq/nw and sq/ne, yshift = -0.75cm] {$f$};
      \node (D) [between = sq/sw and sq/se, yshift = -0.75cm] {$D$};
      \draw[lies over] (f) to (D);
      \draw[>->] (sq/nw) to (f);
      \draw[>->] (sq/sw) to (D);

      \node (Arr/E) [right = of sq/ne] {$\ArrCat{\ECat}$};
      \node (E) [right = of sq/se] {$\ECat$};
      \draw[fibration] (Arr/E) to (E);
    \end{tikzpicture}
  \end{equation}

  We restrict \cref{diag:gl-ret-st:1} along \cref{diag:gl-ret-st:0}:
  \begin{equation}\label[diagram]{diag:gl-ret-st:2}
    \begin{tikzpicture}[diagram,baseline=(sq/sw.base)]
      \SpliceDiagramSquare<sq/>{
        width = 3.5cm,
        nw = h,
        sw = C,
        ne = \pi,
        se = U,
        west/style = {lies over, color = LightGray},
        east/style = lies over,
        south/style = {color = LightGray}
      }
      \node (f) [between = sq/nw and sq/ne, yshift = -0.75cm] {$f$};
      \node (D) [between = sq/sw and sq/se, yshift = -0.75cm] {$D$};
      \node (A) [left = 2.75cm of sq/sw] {$A$};
      \node (B) [left = 2.75cm of D] {$B$};
      \node (C') [left = 2.75cm of A] {$C$};
      \node (D') [left = 2.75cm of B] {$D$};

      \node (B*f) [left = 2.75cm of f] {$B^*f$};
      \node (A*h) [left = 2.75cm of h] {$A^*h$};
      \draw[>->] (A*h) to (B*f);

      \node (f') [left = 2.75cm of B*f] {$f$};
      \node (h') [left = 2.75cm of A*h] {$h$};
      \draw[>->] (h') to (f');

      \draw[->,color=LightGray] (A) to (sq/sw);
      \draw[->] (B) to (D);
      \draw[>->] (A) to (B);

      \draw[->,color=LightGray] (C') to (A);
      \draw[->] (D') to (B);
      \draw[>->] (C') to (D');

      \draw[lies over] (f) to (D);
      \draw[>->] (sq/nw) to (f);
      \draw[>->] (sq/sw) to (D);

      \draw[->] (h') to (A*h);
      \draw[->] (A*h) to (h);
      \draw[double,bend left=30] (h') to (h);

      \draw[lies over] (h') to (C');
      \draw[lies over] (f') to (D');
      \draw[lies over,color=LightGray] (A*h) to (A);
      \draw[lies over] (B*f) to (B);
      \draw[->] (f') to (B*f);
      \draw[->] (B*f) to (f);

      \node (Arr/E) [right = of sq/ne] {$\ArrCat{\ECat}$};
      \node (E) [right = of sq/se] {$\ECat$};
      \draw[fibration] (Arr/E) to (E);
    \end{tikzpicture}
  \end{equation}

  We first glue $\Mor{A^*h}{\pi}$ along $\Mor|>->|{A^*h}{B^*f}$ which lies over $j\in\GlCls$:
  \begin{equation}\label[diagram]{diag:gl-ret-st:3}
    \DiagramSquare{
      nw = A^*h,
      ne = h,
      se = \pi,
      sw = B^*f,
      west/style = >->,
      south/style = exists,
    }
  \end{equation}

  The dotted map of \cref{diag:gl-ret-st:3} restricts along the left-most
  square of $\Mor{f}{B^*f}$ to a solution to our original realignment problem
  (\cref{diag:gl-ret-st:1}):
  \[
    \begin{tikzpicture}[diagram,baseline=(l/sw.base)]
      \SpliceDiagramSquare<l/>{
        nw = h,
        sw = f,
        ne = A^*h,
        se = B^*f,
        west/style = >->,
      }
      \SpliceDiagramSquare<r/>{
        glue = west,
        glue target = l/,
        ne = h,
        se = \pi,
      }
      \draw[double,bend left=30] (l/nw) to (r/ne);
      \draw[->,exists,bend right=30] (l/sw) to (r/se);
    \end{tikzpicture}
    \qedhere
  \]
\end{proof}

\subsection{A small object argument}\label{sec:soa}

In this section we construct a candidate for the generic family
of $\Cls_\kappa$, using a variant of the small object argument. Our
construction is very similar to that of
\citet{shulman:2015:elegant,shulman:2019} but relies on different assumptions.

By \cref{lem:small-set-of-monos}, there is a $\kappa$-small set of monomorphisms
$\GenCls\subseteq\ECat\Sup{\rightarrowtail}$ generating all the monomorphisms of $\ECat$ under
pushout, transfinite composition, and retracts, and whose domains and codomains are $\kappa$-compact.

\begin{definition}\label[definition]{def:realignment-data}
  Let $\Mor[\pi]{E}{U}$ be a relatively $\kappa$-compact map.  A \emph{realignment
  datum} for $\pi$ is defined to be a relatively $\kappa$-compact map $f$
  together with a span of the following form in $\CartArrCat{\ECat}$, in which
  $\Mor|>->|{h}{f}$ lies horizontally over an element of $\GenCls$:
  \[
    \begin{tikzpicture}[diagram]
      \node (h) {$h$};
      \node (f) [left = of nw] {$f$};
      \node (pi) [right = of nw] {$\pi$};
      \path[>->] (h) edge node [sloped,below] {} (f);
      \path[->] (h) edge node [above] {} (pi);
    \end{tikzpicture}
  \]
\end{definition}

There is of course a proper class of realignment data in the sense of
\cref{def:realignment-data}, but \cref{lem:pres-sl-ess-small} ensures that up to
isomorphism there is only a set of realignment data.

\begin{notation}
  We will write $\GlData<\kappa>{\pi}$ for the chosen set of representatives of
  isomorphism classes of realignment data for $\pi$; for any
  $d\in\GlData<\kappa>{\pi}$, we will write
  $\GlSpan{f_d}{h_d}{\pi}$ for the span it represents.
\end{notation}

We record the following lemma for use in \cref{sec:hierarchy}.
\begin{lemma}
  \label[lemma]{lem:small-set-of-realignment-datums}
  Given a strongly inaccessible cardinal $\mu > \kappa$ and a relatively $\kappa$-compact morphism
  $\Mor[\pi]{E}{U}$ such that $U$ is $\mu$-compact, the set $\GlData<\kappa>{\pi}$ is $\mu$-small.
\end{lemma}
\begin{proof}
  Given $\Mor|>->|{A}{B} \in \GenCls$, there is a $\mu$-small set of morphisms $\Mor{B}{U}$ by
  \cref{cor:sharply-larger-facts}. As $\GenCls$ is $\kappa$-small, the conclusion then follows.
\end{proof}

\begin{construction}\label[construction]{con:soa}
  We will define an ideal diagram
  $\Mor[\pi_\kappa^\bullet]{\OrdToCat{\kappa}}{\CartArrCat{\ECat}}$ by well-founded
  induction, finally defining the family $\pi_\kappa:\CartArrCat{\ECat}$ to be
  $\Colim{\OrdToCat{\kappa}}{\pi_\kappa^\bullet}$:
  \[
    \begin{tikzpicture}[diagram]
      \SpliceDiagramSquare<1/>{
        nw/style = pullback,
        ne/style = pullback,
        nw = E_\kappa^0,
        sw = U_\kappa^0,
        west = \pi_\kappa^0,
        ne = \ldots,
        se = \ldots,
        north/style = >->,
        south/style = >->,
      }
      \SpliceDiagramSquare<2/>{
        ne/style = pullback,
        glue = west, glue target = 1/,
        ne = E_\kappa^\alpha,
        se = U_\kappa^\alpha,
        north/style = >->,
        south/style = >->,
      }
      \SpliceDiagramSquare<3/>{
        ne/style = pullback,
        glue = west, glue target = 2/,
        ne = \ldots,
        se = \ldots,
        north/style = >->,
        south/style = >->,
      }
      \SpliceDiagramSquare<4/>{
        glue = west, glue target = 3/,
        ne = E_\kappa,
        se = U_\kappa,
        east = \pi_\kappa,
        north/style = >->,
        south/style = >->,
      }
    \end{tikzpicture}
  \]

  We initialize the iteration by setting $\pi_\kappa^0 \coloneqq
  \InitObj{\CartArrCat{\ECat}}$.  In the successor case, we assume $\pi_\kappa^\alpha \in
  \CartArrCat{\ECat}$ and define $\pi_\kappa^{\alpha+1}$ to be the following pushout
  computed in $\CartArrCat{\ECat}$ using
  \cref{lem:colim-in-cart-arr}.
  \begin{equation}\label[diagram]{diag:soa-succ}
    \DiagramSquare{
      width = 3cm,
      se/style = pushout,
      west/style = >->,
      east/style = {exists,>->},
      south/style = {->,exists},
      nw = \Coprod{d\in\GlData<\kappa>{\pi_\kappa^\alpha}}h_d,
      ne = \pi_\kappa^\alpha,
      sw = \Coprod{d\in\GlData<\kappa>{\pi_\kappa^\alpha}}f_d,
      se = \pi_\kappa^{\alpha+1}
    }
  \end{equation}

  At a limit ordinal $\alpha$, fix an ideal diagram
  $\Mor[\pi_\kappa^\bullet]{\OrdToCat{\alpha}}{\CartArrCat{\ECat}}$ and define
  $\pi_\kappa^\alpha\coloneqq\Colim{\OrdToCat{\alpha}}\pi_\kappa^\bullet$.
\end{construction}

\begin{lemma}\label[lemma]{lem:soa:stages-rel-pres}
  The ideal diagram $\Mor[\pi_\kappa^\bullet]{\OrdToCat{\kappa}}{\CartArrCat{\ECat}}$ from \cref{con:soa} is valued in
  relatively $\kappa$-compact morphisms.
\end{lemma}

\begin{proof}
  We proceed by induction on ordinals $\alpha\leq \kappa$. The base case
  $\pi_\kappa^0 = \InitObj{\CartArrCat{\ECat}}$ is relatively $\kappa$-compact by
  \cref{lem:colim-rel-pres}.
  Next we check that $\pi_\kappa^{\alpha+1}$ is relatively $\kappa$-compact
  assuming $\pi_\kappa^\alpha$ is relatively $\kappa$-compact.  We may
  apply \cref{lem:colim-rel-pres} because \cref{diag:soa-succ} enjoys descent
  as a pushout along a monomorphism, so it suffices to check that each node of
  \cref{diag:soa-succ} is relatively $\kappa$-compact. We have already
  assumed that $\pi_\kappa^\alpha$ is relatively $\kappa$-compact; both
  $\Coprod{d\in\GlData<\kappa>{\pi_\kappa^\alpha}}{f_d}$ and
  $\Coprod{d\in\GlData<\kappa>{\pi_\kappa^\alpha}}{h_d}$ are relatively
  $\kappa$-compact again by \cref{lem:colim-rel-pres} because coproducts
  enjoy descent and both $f_d,h_d$ are relatively $\kappa$-compact as
  pullbacks of $\pi_\kappa^\alpha$.
  In the limit case we assume
  $\pi_\kappa^\beta$ relatively $\kappa$-compact for each $\beta<\alpha$,
  and observe that $\Colim{\OrdToCat{\alpha}}{\pi_\kappa^\bullet}$ is
  relatively $\kappa$-compact by \cref{lem:colim-rel-pres} again, since
  $\OrdToCat{\alpha}$ is a filtered preorder and therefore its diagrams
  enjoy descent (\cref{lem:filtered-colimit-descent}).
\end{proof}

\begin{lemma}
  The transfinite composition $\pi_\kappa\coloneqq \Colim{\OrdToCat{\kappa}}{\pi_\kappa^\bullet}$ is relatively $\kappa$-compact.
\end{lemma}

\begin{proof}
  By \cref{lem:colim-rel-pres,lem:soa:stages-rel-pres} using the fact that
  transfinite compositions enjoy descent (\cref{lem:filtered-colimit-descent}).
\end{proof}

\subsection{Realignment for the universe}

In \cref{sec:soa} we have constructed a relatively $\kappa$-compact map
$\Mor[\pi_\kappa]{E_\kappa}{U_\kappa}$ using the small object argument. We wish
to show that this map exhibits $\Cls_\kappa$ as a universe satisfying
\textbf{(U5,8)}, \ie $\pi_\kappa$ is generic for relatively
$\kappa$-compact maps and satisfies the realignment condition. Because
realignment is stronger than genericity
(\cref{lem:genericity-from-realignment}), we will focus on the former.

We recall from \cref{nota:gl-cls} that $\GlCls{\pi_\kappa}$ denotes the largest
class of monomorphisms relative to which $\prn{\Cls,\pi_\kappa}$ supports
realignment. From \cref{lem:small-set-of-monos} we recall that $\GenCls$ is a set
of monomorphisms generating $\ECat$ under pushout, transfinite composition, and
retracts, and we have assumed that the domain of any $m\in\GenCls$ is
$\kappa$-compact.

\begin{lemma}\label[lemma]{lem:gen-monos-realignable}
  Every generating monomorphism is realignable, \ie we have
  $\GenCls\subseteq\GlCls$.
\end{lemma}

\begin{proof}
  Let $\Mor|>->|[i]{A}{B}$ be an element of $\GenCls$; to check that
  $i\in\GlCls{\pi_\kappa}$, we fix a realignment problem in $\Cls_\kappa$ whose extent lies over
  $\Mor|>->|[i]{A}{B}$.
  \begin{equation}\label[diagram]{diag:gen-monos-realignable:0}
    \begin{tikzpicture}[diagram,baseline=(sq/sw.base)]
      \SpliceDiagramSquare<sq/>{
        width = 3.75cm,
        nw = h,
        sw = A,
        ne = \pi_\kappa,
        se = U_\kappa,
        west/style = lies over,
        east/style = lies over,
        south/style = {color = {LightGray}}
      }
      \node (f) [between = sq/nw and sq/ne, yshift = -1cm] {$f$};
      \node (B) [between = sq/sw and sq/se, yshift = -1cm] {$B$};
      \path[lies over] (f) edge (B);
      \path[>->] (sq/nw) edge (f);
      \path[>->] (sq/sw) edge (B);

      \node (Arr/E) [right = of sq/ne] {$\ArrCat{\ECat}$};
      \node (E) [right = of sq/se] {$\ECat$};
      \path[fibration] (Arr/E) edge (E);
    \end{tikzpicture}
  \end{equation}

  Because $\Mor|>->|{A}{B}\in\GenCls$, we know that $A$ is
  $\kappa$-compact; this is the same as to say that $\Hom[\ECat]{A}{-}$
  commutes with $\kappa$-filtered colimits, in particular the colimit
  $U_\kappa = \Colim{\OrdToCat{\kappa}}{U_\kappa^\bullet}$.
  Thus, using the construction of colimits in the category of sets, there
  exists some $\alpha$ such that $\Mor{h}{\pi_\kappa}$ factors
  through $\Mor|>->|{\pi_\kappa^\alpha}{\pi_\kappa}$; the successor case of the
  small object argument adjoins realignments along generating monomorphisms, so it is
  appropriate to factor our realignment problem like so:
  \begin{equation}\label[diagram]{diag:gen-monos-realignable:1}
    \begin{tikzpicture}[diagram,baseline=(f.base)]
      \node (h) {$h$};
      \node (f) [below = of h] {$f$};
      \node (pi/k/a) [right = of h] {$\pi_\kappa^\alpha$};
      \node (pi/k/a+1) [right = of pi/k/a] {$\pi_\kappa^{\alpha+1}$};
      \node (pi/k) [right = of pi/k/a+1] {$\pi_\kappa$};
      \draw[>->] (h) to (f);
      \draw[->,color=RedDevil] (h) to (pi/k/a);
      \draw[>->] (pi/k/a) to (pi/k/a+1);
      \draw[>->] (pi/k/a+1) to (pi/k);
    \end{tikzpicture}
  \end{equation}

  The intermediate realignment span $\GlSpan{f}{h}{\pi_\kappa^\alpha}$ can be
  represented by a realignment datum $d\in\GlData<\kappa>{\pi_\kappa^\alpha}$. We may
  therefore compose the induced injections to obtain a solution
  $\Mor{f}{\pi_\kappa}$ to the realignment problem \cref{diag:gen-monos-realignable:0}.
  \[
    \begin{tikzpicture}[diagram,baseline=(l/sw.base)]
      \SpliceDiagramSquare<l/>{
        nw = h,
        sw = f,
        ne = \Coprod{d\in\GlData<\kappa>{\pi_\kappa^\alpha}}h_d,
        se = \Coprod{d\in\GlData<\kappa>{\pi_\kappa^\alpha}}f_d,
        west/style = >->,
        east/style = >->,
        north/style = >->,
        south/style = {>->,color = RegalBlue},
        width = 2.5cm,
      }
      \SpliceDiagramSquare<r/>{
        glue = west,
        glue target = l/,
        width = 2.5cm,
        east/style = >->,
        ne = \pi_\kappa^\alpha,
        se = \pi_\kappa^{\alpha+1},
        se/style = pushout,
        south/style = {->,color = RegalBlue},
      }
      \draw[->,color=RedDevil,bend left=30] (l/nw) to (r/ne);
      \node (pi/k) [right = of r/ne] {$\pi_\kappa$};
      \draw[>->] (r/ne) to (pi/k);
      \draw[>->,color=RegalBlue] (r/se) to (pi/k);
    \end{tikzpicture}
    \qedhere
  \]
\end{proof}

\begin{corollary}
  All monomorphisms are realignable, \ie we have $\GlCls{\pi_\kappa} = \ECat\Sup{\rightarrowtail}$.
\end{corollary}

\begin{proof}
  We have assumed that $\GenCls$ generates $\ECat\Sup{\rightarrowtail}$ under
  pushout, transfinite composition, and retracts; but $\GlCls{\pi_\kappa}$ is
  saturated (\cref{sec:saturation}), so our result follows from the fact that
  generating monomorphisms are realignable (\cref{lem:gen-monos-realignable}).
\end{proof}

\begin{corollary}\label[corollary]{cor:rel-cpt-universe-all-axioms}
  $\Cls_\kappa$ is a universe satisfying \textbf{(U1--8)}.
\end{corollary}

\subsection{A cumulative universe hierarchy}
\label{sec:hierarchy}

Fix a second strongly inaccessible cardinal $\mu > \kappa$. We obtain a generic map $\pi_\mu$ for
$\Cls_\mu$ satisfying \textbf{(U8)} by the same small object argument detailed in \cref{sec:soa}.

Genericity of $\pi_\mu$ implies that we automatically obtain a cartesian morphism
$\Mor{\pi_\kappa}{\pi_\mu}$ but this map is not generally a monomorphism. On the other hand, we can
choose our own cartesian monomorphism $\Mor|>->|{\pi_\kappa}{\pi_\mu}$ by means of a pointwise
construction.

\begin{lemma}
  There exists a cartesian monomorphism $\Mor|>->|{\pi_\kappa}{\pi_\mu}$.
\end{lemma}

\begin{proof}
  We recall that each $\pi_\lambda$ is
  $\Colim{\OrdToCat{\kappa}}{\pi_\lambda^\bullet}$.
  Because filtered colimits enjoy descent, by \cref{lem:good-colim-preserves-monos} to construct a cartesian monomorphism
  $\Mor|>->|{\Colim{\OrdToCat{\kappa}}{\pi_\kappa^\bullet}}{\Colim{\OrdToCat{\kappa}}{\pi_\mu^\bullet}}$,
  it suffices to define a cartesian monomorphism of diagrams
  $\Mor|>->,color=RegalBlue|[\ell]{\pi_\kappa^\bullet}{\pi_\mu^\bullet}$:
  \begin{equation}
    \begin{tikzpicture}[diagram]
      \node (pi/k/*) {$\pi_\kappa^\bullet$};
      \node (pi/m/*) [above right = 2.5cm of pi/k/*] {$\pi_\mu^\bullet$};
      \node (pi/m) [below right = 2.5cm of pi/m/*] {$\brc{\pi_\mu}$};
      \draw[>->,exists] (pi/k/*) to (pi/m/*);
      \draw[>->,color=RegalBlue] (pi/k/*) to (pi/m/*);
      \draw[>->] (pi/m/*) to node [sloped,above] {$\eta\Sub{\pi_\mu^\bullet}$} (pi/m);
      \draw[>->,exists] (pi/k/*) to (pi/m);
    \end{tikzpicture}
  \end{equation}

  We construct our natural transformation
  $\Mor|>->|{\pi_\kappa^\bullet}{\pi_\mu^\bullet}$ step-wise; the only
  interesting case is to define
  $\Mor|>->|{\pi_\kappa^{\alpha+1}}{\pi_\mu^{\alpha+1}}$ given
  $\Mor|>->|{\pi_\kappa^\alpha}{\pi_\mu^\alpha}$.
  By \cref{lem:good-colim-preserves-monos} it suffices to define a cartesian
  monomorphism between the defining spans of
  $\pi_\kappa^{\alpha+1},\pi_\mu^{\alpha+1}$, since they are pushouts along
  monomorphisms and hence enjoy descent in $\ArrCat{\ECat}$. Such a morphism is
  trivially induced by the embedding that sends a realignment span
  $\GlSpan{f}{h}{\pi_\kappa^\alpha}$ to $\GlSpan{f}{h}{\pi_\kappa^{\alpha+1}}$
  by postcomposition with
  $\Mor|>->|{\pi_\kappa^\alpha}{\pi_\kappa^{\alpha+1}}$.
\end{proof}

\begin{lemma}
  $U_\kappa$ is $\mu$-compact.
\end{lemma}

\begin{proof}
  We argue that $U_\kappa$ is $\mu$-compact by showing that it is the $\mu$-small colimit of
  $\mu$-small objects. Recall that $U_\kappa = \Colim{\OrdToCat{\kappa}} U_\kappa^\bullet$, so it
  suffices to argue that $U_\kappa^\alpha$ is $\mu$-compact for each $\alpha < \kappa$.

  We show this by transfinite induction on $\alpha < \kappa$. The limit case is immediate:
  $U_\kappa^\alpha$ is then a $\mu$-small colimit of $\mu$-compact objects. Fix $\alpha < \kappa$
  and assume that $U_\kappa^\alpha$ is $\mu$-small. $U_\kappa^{\alpha + 1}$ is defined as the
  following pushout:
  \[
    \DiagramSquare{
      width = 3cm,
      se/style = pushout,
      west/style = >->,
      east/style = {exists,>->},
      south/style = {->,exists},
      nw = \Coprod{d\in\GlData<\kappa>{\pi_\kappa^\alpha}} \Cod{h_d},
      ne = U_\kappa^\alpha,
      sw = \Coprod{d\in\GlData<\kappa>{\pi_\kappa^\alpha}} \Cod{f_d},
      se = U_\kappa^{\alpha+1}
    }
  \]

  By \cref{lem:small-set-of-monos,lem:small-set-of-realignment-datums} together with our assumption that
  $U_\kappa^\alpha$ is $\mu$-compact, this is a $\mu$-small colimit of $\mu$-compact objects so
  $U_\kappa^{\alpha + 1}$ is $\mu$-compact.
\end{proof}

Given a poset $(I, \le)$ and conservative functor $\Mor[\lambda]{I}{\CoSl{\mathbf{Card}}{\kappa}}$
of strongly inaccessible cardinals, these results extend to a hierarchy of universes indexed in $I$:
\begin{corollary}
  \label[corollary]{cor:hierarchy}
  Each universe $S_{\lambda_{i}}$ satisfies \textbf{(U1--8)} and for each $i < j$, there is a
  cartesian monomorphism $\Mor|>->|{\pi_{\lambda_i}}{\pi_{\lambda_j}}$ and $\Cod{\pi_{\lambda_i}}$
  is $\lambda_j$-compact.
\end{corollary}

\section{Relating internal formulations of realignment}\label{sec:internal-formulations}

We have focused on the external formulation of realignment as a property of a
class of maps; recent years have seen several applications of type-theoretic
formulation of realignment that employs the internal language of a topos. In
\cref{sec:internal-realignment} we discuss a logical formulation popularized by
Orton and Pitts, which we compare with a more geometrical formulation due to
Sterling in \cref{sec:recollement} that mirrors the recollement of a space from
open and closed subspaces, completing the latent analogy with Artin gluing.

\subsection{Internal realignment \`a la Orton and Pitts}\label{sec:internal-realignment}

In another guise, \citet{cchm:2017} has employed the realignment property in the
cubical set model of cubical type theory, later rephrased into the internal
language of topoi by \citet{bbcgsv:2016} and employed by
\citet{orton-pitts:2016} to give more abstract and general constructions of
models of cubical type theory in presheaf topoi.

In what follows, we fix a universe $\Cls$ satisfying \textbf{(U1--5)} such that,
in particular, there is a generic map $\Mor[\pi]{E}{U}$ for $\Cls$. We recall
the internal version of the realignment axiom for $U$ below as presented by
\textcite[Axiom 9 (\texttt{ax}\textsubscript{9})]{orton-pitts:2016}, using
informal type theoretic notations.

\begin{notation}
  For any $B:U$, an \emph{isomorph} of $B$ is defined to be a type $A:U$
  together with an isomorphism $f : A \cong B$. We will write $\Iso{\Cls}{B}
  \coloneqq \Sum{A:U}{A\cong B}$ for the type of isomorphs of $B$,
  and $\Iso{\Cls}\coloneqq\Sum{B:U}\Iso{\Cls}{B}$ for the object of
  isomorphisms.
\end{notation}

\begin{notation}
  We will write $X^+$ for the partial map classifier $\Sum{\phi:\Omega}
  X^\phi$, and $\eta^+: X \to X^+$ for its unit.
\end{notation}

\begin{definition}\label[definition]{def:internal-glue}
  A \emph{realignment structure} is defined to
  be an element of the dependent type
  $
    \Prod{B : U}\Prod{A:\Iso{\Cls}{B}^+}
    \Compr{G:\Iso{\Cls}{B}}{
      {A{\downarrow}} \to
      A = \eta^+\prn{G}
    }
  $.
  The \emph{realignment axiom} on $U$ postulates the existence of a realignment
  structure.
\end{definition}

Combining the application described in
\cref{sec:realignment-stc} with the internal perspective of
\citet{orton-pitts:2016}, the realignment operation is included as
an axiom of \emph{synthetic Tait computability}~\citep{sterling:2021:thesis},
the mathematical framework behind the recent normalization result for cubical
type theory~\citep{sterling-angiuli:2021}.

We demonstrate in \cref{lem:external-to-internal,lem:internal-to-external} that
the existence of realignment structures in the sense of
\cref{def:internal-glue} is equivalent to the realignment property of
\cref{def:realignment}.

\begin{notation}
  We will write
  $\Iso{\Cls}^* : \Sl{\ECat}{\Iso{\Cls}}$ for the dependent type
  $I:\Iso{\Cls}\vdash \pi_1\prn{I}$ of pointed isomorphisms.
  We define the type $\Desc{\Cls}$ of $\Cls$-\emph{realignment data} to be the
  dependent sum $\Sum{B:U}\Iso{\Cls}{B}^+$. We will write
  $\Desc{\Cls}^* : \Sl{\ECat}{\Desc{\Cls}}$
  for the dependent type $D:\Desc{\Cls}\vdash \pi_1\prn{D}$ of pointed realignment data.
\end{notation}

\begin{lemma}\label[lemma]{lem:external-to-internal}
  Let $\Cls$ be a universe satisfying \textbf{(U8)} for the class of all
  monomorphisms; then $\Cls$ has a realignment structure.
\end{lemma}

\begin{proof}
  We have a cartesian monomorphism $\Mor|>->|{\Iso{\Cls}^*}{\Desc{\Cls}^*}$
  that turns an isomorphism into the corresponding \emph{total} realignment datum
  with $\phi\coloneqq\top$. Taking the domain of an isomorphism
  corresponds to a cartesian map $\Mor{\Iso{\Cls}^*}{\pi}$.
  Combining these, we may rephrase \cref{def:internal-glue} as the existence of
  a cartesian morphism $\Mor{\Desc{\Cls}^*}{\pi}$ in the following
  configuration:
  \begin{equation}\label[diagram]{diag:iso-desc-lift}
    \begin{tikzpicture}[diagram,baseline=(sw.base)]
      \node (nw) {$\Iso{\Cls}^*$};
      \node (sw) [below = of nw] {$\Desc{\Cls}^*$};
      \node (ne) [right = 2.5cm of nw] {$\pi$};
      \path[->] (nw) edge (ne);
      \path[>->] (nw) edge (sw);
      \path[->,exists] (sw) edge (ne);
    \end{tikzpicture}
  \end{equation}

  The dotted map of \cref{diag:iso-desc-lift} exists by the realignment
  axiom because $\Desc{\Cls}^*\in\Cls$.
\end{proof}

\begin{lemma}\label[lemma]{lem:internal-to-external}
  Suppose that $\Cls$ has a realignment structure; then $\Cls$ satisfies
  \textbf{(U8)} for the class of all monomorphisms.
\end{lemma}

\begin{proof}
  We transform external realignment problems into internal ones. Fix a span of cartesian maps as below such that $f\in \Cls$:
  \begin{equation}
    \begin{tikzpicture}[diagram,baseline=(U.base)]
      \node (h) [pullback, flipped pullback] {$h$};
      \node (f) [left = of h] {$f$};
      \node (pi) [right = of h] {$\pi$};
      \node (Phi) [below = of h] {$\Phi$};
      \node (Gm) [below = of f] {$\Gamma$};
      \node (U) [below = of pi] {$U$};
      \path[->,color=RedDevil] (h) edge node [above] {} (pi);
      \path[>->,color=RegalBlue] (h) edge (f);
      \path[>->,color=RegalBlue] (Phi) edge node [below] {$p_\phi$} (Gm);
      \path[->,color=RedDevil] (Phi) edge node [below] {$A$} (U);
      \path[lies over] (h) edge (Phi);
      \path[lies over] (f) edge (Gm);
      \path[lies over] (pi) edge (U);
    \end{tikzpicture}
  \end{equation}

  Because $f\in \Cls$, we additionally have:
  \begin{equation}
    \begin{tikzpicture}[diagram,baseline=(l/sw.base)]
      \SpliceDiagramSquare<l/>{
        west/style = lies over,
        east/style = lies over,
        south/style = {color=RegalBlue,>->},
        north/style = {color=RegalBlue,>->},
        nw = h,
        ne = f,
        sw = \Phi,
        se = \Gamma,
        south = p_\phi,
        nw/style = pullback,
        ne/style = pullback,
      }
      \SpliceDiagramSquare<r/>{
        glue = west, glue target = l/,
        east/style = lies over,
        se = U,
        ne = \pi,
        south = B,
      }
    \end{tikzpicture}
  \end{equation}

  We take the characteristic map of $\Phi$:
  \begin{equation}
    \DiagramSquare{
      nw = \Phi,
      sw = \Gamma,
      ne = \TermObj{\ECat},
      se = \Omega,
      west = p_\phi,
      east = \top,
      south = \phi,
      west/style = >->,
      south/style = {->,exists},
      nw/style= pullback,
    }
  \end{equation}

  We have a map $\Mor{\Phi}{\Iso{\Cls}{B\circ p_\phi}}$ determined by $A$,
  which we observe forms the base of a cartesian map $\Mor{h}{\Iso{\Cls}^*}$. On
  the other hand, we have a map $\Mor{\Gamma}{\Iso{\Cls}{B}^+}$, \ie a partial
  isomorphism with support $\phi$ between $A$ and $B\circ p_\phi$.
  Therefore we
  have a realignment datum $\Mor{\Gamma}{\Desc{\Cls}}$ determined by $B$ and our
  partial isomorphism; in fact, this is the base of a cartesian map
  $\Mor{f}{\Desc{\Cls}^*}$ which we may compose with the realignment
  structure to obtain the desired factorization:
  \begin{equation}
    \begin{tikzpicture}[diagram,baseline=(sw.base)]
      \node (nw) {$\Iso{\Cls}^*$};
      \node (sw) [below = of nw] {$\Desc{\Cls}^*$};
      \node (nww) [left = of nw] {$h$};
      \node (sww) [left = of sw] {$f$};
      \node (ne) [right = 2.5cm of nw] {$\pi$};
      \path[->] (nw) edge (ne);
      \path[->,color=RedDevil,bend left=30] (nww) edge (ne);
      \path[>->] (nw) edge (sw);
      \path[->] (sw) edge (ne);
      \path[->] (sww) edge (sw);
      \path[>->,color=RegalBlue] (nww) edge (sww);
      \path[->] (nww) edge (nw);
      \path[->,exists,bend right=70] (sww) edge (ne);
    \end{tikzpicture}
  \end{equation}

  In short, we solved the realignment problem by restricting from the \emph{generic} case.
\end{proof}

\subsection{Realignment and recollement}\label{sec:recollement}

Sterling has recently advanced an alternative~\citep{sterling-harper:2022} to
the internal characterization of Orton and Pitts (\cref{sec:internal-realignment})
based on the \emph{recollement} of a sheaf from its components over
complementary open and closed subspaces. We recall the basics of the theory
from SGA~4~\citep{sga:4}.

When $\XCat$ is a topos, a subterminal object $\Mor|>->|{J}{\TermObj{\XCat}}$
corresponds to an open subtopos $\Sl{\XCat}{J}$ such that the open inclusion
geometric morphism $\Mor|open immersion|[j_*]{\Sl{\XCat}{J}}{\XCat}$ is the
right adjoint to the pullback functor $\Mor[j^*]{\XCat}{\Sl{\XCat}{J}}$ that sends
$E$ to $\Mor{E\times J}{J}$. Meanwhile we may form the complementary
\emph{closed} subtopos $\ClSubcat{\XCat}{U} = \XCat\setminus\Sl{\XCat}{J}$ by
considering the subcategory of $\XCat$ spanned by objects $E$ for which the
canonical map $\Mor{E\times J}{J}$ is an isomorphism. The closed inclusion
$\Mor|closed immersion|[i_*]{\ClSubcat{\XCat}{J}}{\XCat}$ then has a left exact
left adjoint $\Mor[i^*]{\XCat}{\ClSubcat{\XCat}{J}}$ taking $E$ to the join
$E\star J$, \ie the following pushout:
\[
  \DiagramSquare{
    nw = E\times J,
    sw = E,
    ne = J,
    se = E\star J,
    se/style = dotted pushout,
    nw/style = dotted pullback,
    west/style = >->,
    south/style = {exists,->},
    east/style = {exists,>->},
  }
\]

The Grothendieck school then develops both a \emph{global} and a \emph{local}
recollement theory for the open-closed partition
$\prn{\Sl{\XCat}{J},\ClSubcat{\XCat}{J}}$ of $\XCat$:

\subsubsection{Global recollement~\citep{sga:4}}\label{sec:global-recollement}
The topos $\XCat$ may be reconstructed from its open and closed subtopoi as
the comma category $\ClSubcat{\XCat}{J}\downarrow i^*j_*$, \ie the Artin gluing of $i^*j_*$. In other
words, the diagram below is pseudocartesian in the (very large) bicategory of all
categories, in which the upper functor
$\Mor[q]{\XCat}{\ArrCat{\ClSubcat{\XCat}{J}}}$ sends an object $E$ to the
morphism $i^*\prn{\Mor[\eta_E]{E}{j_*j^*E}}$ in $\ClSubcat{\XCat}{J}$.
\[
  \DiagramSquare{
    nw = \XCat,
    ne = \ArrCat{\ClSubcat{\XCat}{J}},
    se = \ClSubcat{\XCat}{J},
    east = \Cod\Sub{\ClSubcat{\XCat}{J}},
    south = i^*j_*,
    west = j^*,
    sw = \Sl{\XCat}{J},
    north = q,
    nw/style = pullback,
  }
  \qedhere
\]
\endproof

From the global recollement of the topos $\XCat$ from its open and closed
subtopoi, the Grothendieck school concludes a \emph{local} recollement or
\emph{fracture theorem} that reconstructs an object of the topos from its
components over the open and closed subtopoi.\footnote{Such a fracture theorem
is developed in much greater generality for left exact modalities by
\citet{rijke-shulman-spitters:2020}.}

\subsubsection{Local recollement~\citep{sga:4}}\label{sec:local-recollement}

Under the same assumptions, any object $E$ of $\XCat$ may be reconstructed from
its restrictions $j^*E,i^*E$ to the open and closed subtopoi respectively. In
particular, the following diagram is cartesian in $\XCat$:
\[
  \DiagramSquare{
    nw = E,
    ne = i_*i^*E,
    sw = j_*j^*E,
    se = i_*i^*j_*j^*E,
    north = \eta_E,
    west = \eta_E,
    east = {i_*i^*\eta_E = \mathrlap{i_*qE}},
    south = \eta\Sub{j_*j^*E},
    nw/style = pullback,
    width = 3cm,
  }
  \qedhere
\]
\endproof

The above follows immediately from the global recollement
(\cref{sec:global-recollement}); conversely, if $O:\Sl{\XCat}{J}$ is an object of the open
subtopos and $\Mor[p]{K}{i^*O}:\ClSubcat{\XCat}{J}$ is a family of objects in the closed subtopos,
then the pullback of the latter along $\Mor{O}{i_*j^*O}$ in $\XCat$ is a
morphism $\Mor{E}{j_*O}$ that is \emph{isomorphic} to the unit
$\Mor{E}{j_*j^*E}$:
\begin{equation}\label[diagram]{diag:recollement:KO}
  \begin{tikzpicture}[diagram,baseline=(sw.base)]
    \SpliceDiagramSquare{
      nw = E,
      ne = i_*K,
      sw = j_*O,
      se = i_*i^*j_*O,
      west = \eta\Sub{j_*O}^*\prn{i_*p},
      east = i_*p,
      south = \eta\Sub{j_*O},
      nw/style = pullback,
      west/node/style = upright desc,
      width = 3cm,
      height = 2.5cm,
    }
    \node (jE) [between = nw and sw, xshift=-2cm] {$j_*j^*E$};
    \draw[->] (nw) to node [sloped,above] {$\eta_E$} (jE);
    \draw[->] (jE) to node [sloped,below] {$\cong$} (sw);
  \end{tikzpicture}
\end{equation}

\begin{question}\label[question]{question:realignment}
  Can $E$ be chosen in \cref{diag:recollement:KO} to make the
  isomorphism $\Mor{j_*j^*E}{j_*O}$ an \emph{identity} map?
\end{question}

Although identity of objects is not properly part of the language of category
theory, it becomes meaningful when considering \emph{internal categories} as we
do in \cref{sec:internal-recollement} below. We will see that the realignment
axiom \textbf{(U8)} for a full internal subtopos corresponds to the ability to
construct a version of \cref{diag:recollement:KO} in which $j_*j^*E = j^*O$
strictly.

\subsubsection{Internal recollement}\label{sec:internal-recollement}

Let $\mathcal{U}$ be a universe in $\XCat$ and let $\Mor[p]{E}{U}$ be a generic
family for $\mathcal{U}$; then $U$ constitutes a \emph{full internal subtopos} of
$\XCat$ in the sense of \citet{benabou:1973}. Consequently we may think of
$\mathcal{U}$ as a topos $C^*U$ in every slice $\Sl{\XCat}{C}$ of $\XCat$; hence
any monomorphism $\Mor|>->|{J}{C}$ in $\XCat$ corresponds to a subterminal
object in $\Sl{\XCat}{C}$, \ie an open subtopos of $C^*U$. Therefore we may
replay the global and local recollement for each $C^*U$ using the same
constructions.

Letting $J:\Omega$ be a proposition in $\XCat$, we note that the exponential
family $\Mor{E^J}{U^J}$ is generic for the \emph{open} subtopos of $U$
determined by the proposition $J$. We will write $\Mor[J_*]{U^J}{U}$ for the
function that sends a family $O:U^J$ to its dependent product $\Prod{z:J}Oz$;
the left adjoint $\Mor[J^*]{U}{U^J}$ takes a type $A$ to the constant family
$\lambda\_:J.A$.
Likewise we may obtain a generic family for the \emph{closed}
subtopos by considering the subobject $\ClSubcat{U}{J} \subseteq U$ spanned
by types $A$ such that $\Mor{p\brk{A}\times J}{J}$ is an isomorphism;
following \citet{rijke-shulman-spitters:2020}, we will refer to such types as
$J$-connected.

We may now revisit our \cref{question:realignment} concerning
\cref{diag:recollement:KO} in the internal language. Let $O:U^J$ be an object
of the open subtopos and let
$K:J_*O\to\ClSubcat{U}{J}$ be a family of $J$-connected
objects. Then an affirmative answer to \cref{question:realignment} would
produce some $E:U$ together with an isomorphism
$f_E:\prn{\Sum{x:J_*O}Kx}\to E$ in $U$ such that $j^*E = O$ strictly and $j^*f_E$ is strictly equal to
$\lambda z:J. \lambda \prn{x,y}. xz$. In other words, we are asking for a type
constructor $\Con{Glue}$ on $U$ with the following interface:
\begin{align*}
  \Con{Glue} &: \Prod{J:\Omega}\Prod{O:U^J} \Prod{K:{J_*O}\to\ClSubcat{U}{J}}\Compr{G:U}{\forall z:J. G = Oz}\\
  \Con{glue} &: \Prod{J:\Omega}\Prod{O:U^J}\Prod{K:{J_*O}\to\ClSubcat{U}{J}}
  \Compr{
    f:\prn{\Sum{x:J_*O}Kx}
    \cong
    \Con{Glue}\,O\,K
  }{
    \forall z:J.\forall x,y.
    f\prn{x,y} = xz
  }
\end{align*}

It is not difficult to verify that the existence of such a type constructor is
equivalent to the internal realignment axiom discussed in \cref{sec:internal-realignment}.

\begin{lemma}
  Let $G$ be a realignment structure for $U$ in the sense of
  \cref{def:internal-glue}; then there exists a $\Con{Glue}$ connective
  satisfying the described rules.
\end{lemma}

\begin{proof}
  Let $O,K$ as above and consider the application of $G$ to $B\coloneqq
  \Sum{x:J_*O}Kx$ and the partial isomorphism $z:J\vdash B \cong Oz$,
  which exists because each fiber of $K$ is $J$-connected. From this pair we
  thus obtain both $\Con{Glue}\,J\,O\,K$ and $\Con{glue}\,J\,O\,K$.
\end{proof}

\begin{lemma}
  Conversely, suppose that we have a $\Con{Glue}$ connective in the sense
  described above; then there exists a realignment structure in the sense of \cref{def:internal-glue}.
\end{lemma}

\begin{proof}
  Given a type $B$ and a partial isomorph $\prn{J,A}:\Iso{\mathcal{U}}{B}^+$,
  we let $O \coloneqq \lambda z:J.\pi_1\prn{Az}$ and $K\coloneqq\lambda
  x:J_*O.\Compr{y:B}{\forall z:J.\prn{\pi_2\prn{Az}}\prn{xz} = y}$. Then we consider the total
  isomorph given by the pair $\prn{\Con{Glue}\,J\,O\,K,\pi_2\circ \prn{\Con{glue}\,J\,O\,K}^{-1}}$.
\end{proof}

The benefit of the present axiomatization is that a family of types being
fiberwise $J$-connected is a \emph{property}; in contrast, the Orton--Pitts
axiomatization (\cref{def:internal-glue}) requires every use of realignment to be accompanied by a
chosen isomorphism. We have gained significant experience with both
axiomatizations in the context of synthetic Tait
computability~\citep{sterling:2021:thesis,sterling-harper:2021,sterling-angiuli:2021,gratzer:2022:lics,sterling-harper:2022,niu-sterling-grodin-harper:2022},
and found that the present one is substantially simpler to use in practice.

\section{Applications of realignment}
\label{sec:examples}

An immediate consequence of \cref{sec:universe} is an interpretation of
Martin-L{\"o}f type theory with a cumulative hierarchy of universes in arbitrary
Grothendieck topoi (recall that we have assumed a hierarchy of Grothendieck
universes). In fact, the new interpretation of Martin-L{\"o}f type
theory in Grothendieck topoi enables more direct independence proofs for various
axioms such as Markov's principle. But the realignment property itself has
played an important role in the semantics of homotopy type theory as developed
by \citet{awodey:2021:qms}, Kapulkin, Lumsdaine, and Voevodsky~\citep{kapulkin-lumsdaine:2021}, \citet{streicher:2014:simplicial,stenzel:2019:thesis,shulman:2015:elegant,shulman:2019}.
In particular, realignment appears to be a necessary ingredient for
constructing a fibrant and univalent universe.  The same principle is employed
by \citet[Lemma~5.33]{sterling-angiuli-gratzer:2022} in their proof of
\emph{canonicity} for XTT, a variant of cubical type theory: in particular,
\opcit used a special case of \textbf{(U8)} to realign codes in the universe of
an Artin gluing over chosen codes in the universe of its open subtopos.

\subsection{Independence results for Martin-L\"of type theory}\label{sec:independence}

Sheaf semantics has historically been employed to prove independence results
for various forms of logic; the use of sheaf semantics to verify the analogous
results for dependent type theory with universes has been hampered by the
(now-resolved) difficulties in constructing well-behaved universes in sheaf
topoi. These difficulties have motivated two somewhat less direct methods for
proving independence results: constructing \emph{operational} or
\emph{relational} models of type theory using the Beth--Kripke--Joyal sheaf
semantics of predicate logic~\citep{coquand-mannaa:2016}, or by constructing
denotational models of type theory in \emph{stacks} rather than
sheaves~\citep{coquand-mannaa-ruch:2017}. The present work provides a more
direct approach, as the presence of universes validating \textbf{(U1--8)}
ensures a simple and direct denotational semantics of dependent type theory in
sheaves. We illustrate this through a concrete example and sketch a simpler
proof of the independence of Markov's principle.

\subsubsection{Independence of Markov's principle}

Markov's principle states that for any decidable property $P\prn{x}$ of natural
numbers, the proposition $\exists x.P{x}$ is $\lnot\lnot$-stable:
\[
  \forall P : \mathbb{N}\to \mathbf{2}.
  {\lnot\lnot\exists x. Px = 0}
  \to
  \exists x. Px = 0
\]

Formalized in the language of dependent type theory, Markov's principle is
rendered by \citet{coquand-mannaa:2016} equivalently as the existence of a
global element of the following type:
\[
  \Prod{P:\mathbb{N}\to\mathbf{2}}
  \prn{
    \lnot\lnot{\Sum{x:\mathbb{N}} Px = 0}
    \to
    \Sum{x:\mathbb{N}}Px=0
  }
\]

The independence of Markov's principle from intuitionistic higher-order logic
is established easily by considering the internal logic of the topos of sheaves
on Cantor space $\mathcal{C}$, \ie the space of infinite binary sequences
equipped with the product topology. If $\Sh{\mathcal{C}}$ did not model
universes, we would not however be able to use it directly to verify the
independence of Markov's principle from Martin-L\"of type theory with
universes.
Our result concerning universes in Grothendieck topoi, however, allows one to
immediately deduce the independence of Markov's principle from Martin-L\"of
type theory with universes without needing to pass to the significantly more
complex stack semantics of \citet{coquand-mannaa-ruch:2017}, bypassing as well
the detour through operational semantics of \citet{coquand-mannaa:2016}.

\begin{corollary}
  Neither Markov's principle nor its negation is derivable in Martin-L\"of type theory
  with a cumulative hierarchy of strict universes.
\end{corollary}

\subsection{Semantics of the univalent universes}\label{sec:univalence}

The semantics of univalent universes has proved to be a crucial technical
difficulty in models of homotopy type theory and cubical type theory; in
paticular, it is necessary to translate facts between the language of model
category theory and the language of universes.  We briefly illustrate how
judicious application of \textbf{(U8)} has been used in the literature to
entirely eliminate these
difficulties~\citep{streicher:2014:simplicial,kapulkin-lumsdaine:2021,shulman:2015:elegant,shulman:2019,awodey:2021:qms}.
In fact, this observation was the original motivation for
\citet{shulman:2015:elegant} to isolate \textbf{(U8)}.

We illustrate the utility of \textbf{(U8)} by tracing through the salient
aspects of the model given by Kapulkin, Lumsdaine, and Voevodsky~\citep{kapulkin-lumsdaine:2021} and defer to
\citet{shulman:2015:elegant,shulman:2019} for a more systematic approach. Concretely, we
will work in $\SSET$ and fix a pair of strongly inaccessible cardinals
$\kappa_0 < \kappa_1$ inducing universes $\mathcal{V}_0 \subseteq \mathcal{V}_1$
each satisfying \textbf{(U1--8)}. Moreover, by \cref{sec:hierarchy}, we can
choose a generic map for $\mathcal{V}_0$ whose base lies in $\mathcal{V}_1$.

Let $\mathcal{U}_i \subseteq \mathcal{V}_i$ be the class of Kan fibrations in
$\mathcal{V}_i$.

\begin{lemma}
  The class of maps $\mathcal{U}_i$ satisfies \textbf{(U1,3,4,8)}.
\end{lemma}
\begin{proof}
  \textbf{(U1,3)} follow immediately from the fact that $\mathcal{V}_i$
  satisfies \textbf{(U1,3)} and that any right-orthogonal class is closed under
  composition and pullback. \textbf{(U4)} is an immediate consequence of the
  right-properness of the Kan-Quillen model structure.

  To show that $\mathcal{U}_i$ satisfies \textbf{(U8)}, we being by fixing a
  generic family
  $\Mor[\pi\Sub{\mathcal{V}_i}]{E\Sub{\mathcal{V}_i}}{U\Sub{\mathcal{V}_i}}$ for
  $\mathcal{V}_i$ and defining the following restriction of
  $U\Sub{\mathcal{V}_i}$:
  \[
    U\Sub{\mathcal{U}_i} =
    \Compr{
      X : U\Sub{\mathcal{V}_i}
    }{
      X \text{ is a Kan complex}
    }
  \]

  More precisely, a point $\Mor[\alpha]{\Simplex{n}}{U\Sub{\mathcal{V}_i}}$
  factors through $U\Sub{\mathcal{U}_i}$ if $\pi^*\prn{\alpha}$ is a Kan
  fibration. This is a well-defined simplicial set because Kan fibrations are stable
  under pullback.
  We define $\pi\Sub{\mathcal{U}_i}$ (resp.\ $E\Sub{\mathcal{U}_i}$) as the
  restriction of $\pi\Sub{\mathcal{V}_i}$ (resp.\ $E\Sub{\mathcal{V}_i}$) to
  $U\Sub{\mathcal{U}_i}$. We first prove that
  $\pi\Sub{\mathcal{U}_i} \in \mathcal{U}_i$, and then verify
  \textbf{(U8)}.

  By \textbf{(U1)} we conclude that $\pi\Sub{\mathcal{U}_i}$ lies in $\mathcal{V}_i$, and it is moreover
  a Kan fibration almost by definition. Fix a commutative diagram of the
  following shape:
  \[
    \DiagramSquare{
      ne = E\Sub{\mathcal{U}_i},
      se = U\Sub{\mathcal{U}_i},
      nw = \Horn{n}{i},
      sw = \Simplex{n},
      south = \alpha,
      east = \pi\Sub{\mathcal{U}_i},
      east/style = fibration,
      west/style = >->,
    }
  \]

  By definition of $\pi\Sub{\mathcal{U}_i}$, pulling back along $\alpha$ yields
  a Kan fibration, whereby we obtain the necessary lift:
  \[
    \begin{tikzpicture}[diagram]
      \SpliceDiagramSquare{
        width = 3cm,
        ne = E\Sub{\mathcal{U}_i},
        se = U\Sub{\mathcal{U}_i},
        nw = E\Sub{\mathcal{U}_i} \times\Sub{U\Sub{\mathcal{U}_i}} \Simplex{n},
        sw = \Simplex{n},
        nw/style = pullback,
        east/style = fibration,
        west/style = fibration,
      }
      \node[left = 3.25cm of nw] (Horn) {$\Horn{n}{i}$};
      \node[left = 3.25cm of sw] (Simplex) {$\Simplex{n}$};
      \draw[>->] (Horn) to (Simplex);
      \draw[double] (Simplex) to (sw);
      \draw[->, exists] (Horn) to (nw);
      \draw[->, exists] (Horn) to node[upright desc] {$\exists !$} (nw);
      \draw[->, exists] (Simplex) to (nw);
    \end{tikzpicture}
  \]

  Consequently, $\pi\Sub{\mathcal{U}_i} \in \mathcal{U}_i$.
  It remains to show that $\pi\Sub{\mathcal{U}_i}$ satisfies
  \textbf{(U8)}. Accordingly, fix a pair of cartesian squares
  $\Mor[\alpha]{f}{\pi\Sub{\mathcal{U}_i}}$ and $\Mor|>->|[i]{f}{g}$. We apply
  \textbf{(U8)} for $\mathcal{V}_i$ to obtain a cartesian square
  $\Mor[\beta]{g}{\pi\Sub{\mathcal{V}_i}}$ fitting into the following
  commutative diagram:
  \[
    \begin{tikzpicture}[diagram]
      \node (Horn) {$f$};
      \node[below = of Horn] (Simplex) {$g$};
      \node[right = of Horn] (U) {$\pi\Sub{\mathcal{U}_0}$};
      \node[right = of U] (V) {$\pi\Sub{\mathcal{V}_0}$};
      \path[->] (Horn) edge node[above] {$\alpha$} (U);
      \path[>->] (Horn) edge (Simplex);
      \path[>->] (U) edge (V);
      \path[->,exists] (Simplex) edge node[sloped,below] {$\beta$} (V);
    \end{tikzpicture}
  \]

  To complete the proof, it suffices to show that $\beta$ factors through
  $\pi\Sub{\mathcal{U}_i}$ \ie{} that for any cartesian square $\Mor{h}{g}$ such
  that $h$ has a representable base, $h$ is a Kan fibration. This, however,
  follows immediately because $g$ is a Kan fibration.
\end{proof}

We recall a purely homotopy-theoretic fact, referred to by
\citet{awodey:2021:qms} as the \emph{fibration extension property}.
\begin{lemma}
  \label[lemma]{lem:extend-fibration}
  Given a Kan fibration $\FibMor[f]{X}{A}$ and a trivial cofibration
  $\Mor|>->|[i]{A}{B}$, there is a Kan fibration $\FibMor[g]{Y}{B}$ such that
  $i^*g = f$. Additionally, if $f\in\mathcal{V}_i$ then $g\in\mathcal{V}_i$.
\end{lemma}
This result is proved by Kapulkin, Lumsdaine, and Voevodsky~\citep{kapulkin-lumsdaine:2021} using Quillen's theory
of \emph{minimal fibrations}. An alternative approach is given by
\citet[\href{https://kerodon.net/tag/00ZS}{Tag 00ZS}]{kerodon} using Kan's
$\mathsf{Ex}_\infty$ functor. A near immediate consequence of \cref{lem:extend-fibration} and \textbf{(U8)} is the fibrancy of the $U\Sub{\mathcal{U}_0}$:

\begin{theorem}
  \label[theorem]{thm:univ-is-fibrant}
  The object $U\Sub{\mathcal{U}_0}$ lies within $\mathcal{U}_1$.
\end{theorem}

\begin{proof}
  As a subobject of $\mathcal{V}_0$, \textbf{(U2)} implies that
  $U\Sub{\mathcal{U}_0}$ lies within $\mathcal{V}_1$, so it suffices to show
  that $U\Sub{\mathcal{U}_0}$ is a Kan complex. Accordingly, we fix a lifting
  problem for $U\Sub{\mathcal{U}_0}$:
  \[
    \begin{tikzpicture}[diagram]
      \node (Horn) {$\Horn{n}{i}$};
      \node[below = of Horn] (Simplex) {$\Simplex{n}$};
      \node[right = of Horn] (U) {$U\Sub{\mathcal{U}_0}$};
      \path[->] (Horn) edge node[above] {$\alpha$} (U);
      \path[>->] (Horn) edge (Simplex);
    \end{tikzpicture}
  \]

  We must extend $\alpha$ along the inclusion
  $\Mor|>->|{\Horn{n}{i}}{\Simplex{n}}$. We begin by pulling back
  $\pi\Sub{\mathcal{U}_0}$ along $\alpha$, obtaining a Kan fibration
  $\FibMor{\brk{\alpha}}{\Horn{n}{i}}$ and a cartesian map
  $\Mor[h]{\brk{\alpha}}{\pi\Sub{\mathcal{U}_0}}$. Applying
  \cref{lem:extend-fibration}, we can extend $\brk{\alpha}$ along
  $\Mor|>->|{\Horn{n}{i}}{\Simplex{n}}$ to another Kan fibration
  $\FibMor{\brk{\beta}}{\Simplex{n}}$. Next, we apply \textbf{(U8)} to extend
  $h$ along the induced cartesian monomorphism
  $\Mor|>->|{\brk{\alpha}}{\brk{\beta}}$:
  \[
    \begin{tikzpicture}[diagram]
      \node (Horn) {$\brk{\alpha}$};
      \node[below = of Horn] (Simplex) {$\brk{\beta}$};
      \node[right = of Horn] (U) {$\pi\Sub{\mathcal{U}_0}$};
      \path[->] (Horn) edge node[above] {$h$} (U);
      \path[>->] (Horn) edge (Simplex);
      \path[->,exists] (Simplex) edge node [sloped,below] {$\beta$} (U);
    \end{tikzpicture}
  \]

  The downstairs component of $\Mor[\beta]{\brk{\beta}}{\pi\Sub{\mathcal{U}_0}}$ then solves the
  original lifting problem.
\end{proof}

Notice that in the above proof, the application of \textbf{(U8)} allows us to
rephrase a property of the generic family (``$U\Sub{\mathcal{U}_0}$ is a Kan
complex'') as a property of the class of maps $\mathcal{U}_0$ (``Kan fibrations extend along
trivial cofibrations'') to which the standard tools of homotopy theory apply. While
the setup is more complex, the same is true of the proof that
$\pi\Sub{\mathcal{U}_0}$ is univalent. Prior to discussing the proof of
univalence, we must fix a few definitions.

\begin{definition}
  Given Kan fibrations $\FibMor{E_0,E_1}{B}$, we define $\FibMor{\Equiv{E_0}{E_1}}{B}$ to be
  the fibration of weak equivalences between $E_0$ and $E_1$, \ie the subobject of the
  local exponential $\Mor{E_1^{E_0}}{B}$ spanned by weak equivalences.
\end{definition}

Explicitly, a simplex $\Mor[\alpha]{\Simplex{n}}{E_1^{E_0}}$ factors through
$\Equiv{E_0}{E_1}$ if the corresponding morphism
$\Mor{\alpha^*E_0}{\alpha^*E_1}$ over $\Simplex{n}$ is a weak equivalence. In
fact, a map $\Mor{X}{\Equiv{E_0}{E_1}}$ is determined by a pair of maps
$\Mor[f_i]{X}{B}$ along with a weak equivalence $\Mor{f_0^*E_0}{f_1^*E_1}$ over
$X$.

We have avoided a number of subtle points in this definition \eg{}, that weak
equivalences between fibrations are stable under pullback to show that it is
well-defined. These are addressed thoroughly by Kapulkin, Lumsdaine, and
Voevodsky~\citep{kapulkin-lumsdaine:2021}. See \citet{shulman:2015:elegant} for
a less analytic definition of the object of equivalences.

Given a Kan fibration $\FibMor{X}{B}$, we define
$\Mor[\gls{\partial_0,\partial_1}]{\Eq{X}}{B \times B}$ to be
$\Equiv{\pi_1^*X}{\pi_2^*X}$, \ie the object of equivalences between two specified fibers
of $X$. We observe that there is a canonical monomorphism
$\Mor|>->|[\delta_X]{B}{\Eq{X}}$ lying over the diagonal map $\Mor|>->|{B}{B
\times B}$ sending $b:B$ to the identity equivalence
$\Mor{X\brk{b}}{X\brk{b}}$:
\[
  \DiagramSquare{
    nw = B,
    sw = B,
    west/style = double,
    north/style = >->,
    south/style = >->,
    ne = \Eq{X},
    se = B\times B,
    north = \delta_X,
    south = \delta,
    east = \gls{\partial_0,\partial_1},
  }
\]

\begin{definition}
  A Kan fibration $\FibMor{X}{B}$ is called \emph{univalent} when
  $\Mor|>->|[\delta_X]{B}{\Eq{X}}$ is a trivial cofibration.
\end{definition}

We will now sketch the proof that $\pi\Sub{\mathcal{U}_0}$ is univalent.  Just
as with \cref{thm:univ-is-fibrant}, the proof decomposes into two pieces: a
homotopy-theoretic result and a careful analysis and application of
\textbf{(U8)} to parlay this result into the appropriate result on the
universe. For univalence, the relevant homotopy-theoretic fact is the
\emph{equivalence extension property}, apparently first isolated by Kapulkin, Lumsdaine, and Voevodsky~\citep{kapulkin-lumsdaine:2021}, named by Awodey,
and further developed by several authors including Awodey, Coquand, Sattler, and Shulman~\citep{shulman:2015:elegant,shulman:2019,sattler:2017,awodey:2021:qms,cchm:2017}.

\subsubsection{Lemma (Equivalence Extension Property)}\label[lemma]{lem:eep}
\begingroup\itshape
  We consider a diagram of the following shape, in which the downward maps are Kan fibrations, $\Mor|>->|[i]{A}{B}$ is a cofibration, and $\Mor[w]{X}{i^*Y}$ is a weak equivalence:
  \begin{equation}\label[diagram]{diag:eep:0}
    \begin{tikzpicture}[diagram,baseline=(B.base)]
      \node (X) {$X$};
      \node[below right = 0.75cm and 2cm of X] (iY) {$i^*Y$};
      \node[below right = 2.5cm and 1cm of X] (A) {$A$};
      \node[right = 4cm of A] (B) {$B$};
      \node[right = 4cm of iY] (Y) {$Y$};
      \path[>->] (A) edge node[below] {$i$} (B);
      \path[->] (X) edge node[upright desc] {$w$} (iY);
      \path[fibration] (X) edge (A);
      \path[fibration] (iY) edge (A);
      \path[fibration] (Y) edge (B);
      \path[>->] (iY) edge (Y);
    \end{tikzpicture}
  \end{equation}

  Then \cref{diag:eep:0} can be extended to a diagram of the following shape,
  in which $\Mor[\bar{w}]{\bar{X}}{Y}$ is a weak equivalence and
  $\Mor|fibration|{\bar{X}}{B}$ is a fibration, and all three squares are cartesian:
  \[
    \begin{tikzpicture}[diagram]
      \node (X) {$X$};
      \node[below right = 0.75cm and 2cm of X] (iY) {$i^*Y$};
      \node[below right = 2.5cm and 1cm of X] (A) {$A$};
      \node[right = 4cm of X] (Xext) {$\bar{X}$};
      \node[right = 4cm of A] (B) {$B$};
      \node[right = 4cm of iY] (Y) {$Y$};
      \path[>->] (A) edge node[below] {$i$} (B);
      \path[->] (X) edge node[upright desc] {$w$} (iY);
      \path[fibration] (X) edge (A);
      \path[fibration] (iY) edge (A);
      \path[fibration] (Y) edge (B);
      \path[>->] (iY) edge (Y);
      \path[fibration,color=LightGray] (Xext) edge (B);
      \path[->] (Xext) edge node[upright desc] {$\bar{w}$} (Y);
      \path[->] (X) edge (Xext);
    \end{tikzpicture}
  \]

  Moreover, if $\FibMor{X}{A}$ and $\FibMor{Y}{B}$ both belong to
  $\mathcal{U}_0$, so does $\FibMor{\bar{X}}{B}$.
\endgroup

\begin{theorem}
  The family $\Mor[\pi\Sub{\mathcal{U}_0}]{E\Sub{\mathcal{U}_0}}{U\Sub{\mathcal{U}_0}}$ is univalent.
\end{theorem}

\begin{proof}
  Unfolding definitions, we must show that
  $\Mor|>->|[\delta\Sub{E\Sub{\mathcal{U}_0}}]{U\Sub{\mathcal{U}_0}}{\Eq{E\Sub{\mathcal{U}_0}}}$
  is a trivial cofibration; as it is already a cofibration, it is enough to
  check that it is a weak equivalence. Consider \cref{diag:univalence:0} below
  exhibiting $\delta\Sub{E\Sub{\mathcal{U}_0}}$ as a section of the fibration
  $\FibMor[\partial_1]{\Eq{E\Sub{\mathcal{U}_0}}}{U\Sub{\mathcal{U}_0}}$:
  \begin{equation}\label[diagram]{diag:univalence:0}
    \begin{tikzpicture}[diagram,baseline=(se.base)]
      \node (nw) {$U\Sub{\mathcal{U}_0}$};
      \node (ne) [right = 2.5cm of nw] {$\Eq{E\Sub{\mathcal{U}_0}}$};
      \node (se) [below = of ne] {$U\Sub{\mathcal{U}_0}$};
      \draw[double] (nw) to (se);
      \draw[>->] (nw) to node [above] {$\delta\Sub{E\Sub{\mathcal{U}_0}}$} (ne);
      \draw[fibration] (ne) to node [right] {$\partial_1$} (se);
    \end{tikzpicture}
  \end{equation}

  By the 2-out-of-3 property of weak equivalances, it therefore suffices to show that fibration
  $\Mor|fibration|[\partial_1]{\Eq{E\Sub{\mathcal{U}_0}}}{U\Sub{\mathcal{U}_0}}$
  is a \emph{trivial} fibration.
  To this end we fix a cofibration
  $\Mor|>->|{A}{B}$ to check the right lifting property for $\partial_1$:
  \begin{equation}\label[diagram]{diag:univalence:1}
    \DiagramSquare{
      width = 3.5cm,
      nw = A,
      sw = B,
      ne = \Eq{E\Sub{\mathcal{U}_0}},
      se = U\Sub{\mathcal{U}_0},
      south = \bar{\alpha},
      north = \prn{\beta, \alpha, w},
      east = \partial_1,
      west/style = >->,
      east/style = fibration,
    }
  \end{equation}

  In \cref{diag:univalence:1} above, we have written $\beta,\alpha$ for the two
  codes $\Mor{A}{U\Sub{\mathcal{U}_0}}$ and
  $\Mor[w]{\brk{\beta}}{\brk{\alpha}}$ for the weak equivalence between the
  corresponding fibers of $\pi\Sub{\mathcal{U}_0}$, writing $\brk{\alpha}$ for
  the pullback of $\pi\Sub{\mathcal{U}_0}$ along $\alpha$, \etc; then
  $\bar\alpha$ is an \emph{extension} of the code $\alpha$ along the
  cofibration $\Mor|>->|{A}{B}$. Our goal is to provide similar extensions
  of $\beta,w$ to produce an equivalence between $B$-valued fibers of
  $\pi\Sub{\mathcal{U}_0}$.
  Considering the fiber of
  $\pi\Sub{\mathcal{U}_0}$ at $\bar\alpha$, we have a Kan fibration
  $\FibMor{\brk{\bar\alpha}}{B}$ whose pullback along $\Mor|>->|{A}{B}$ is
  $\FibMor{\brk{\alpha}}{A}$. We summarize the situation as follows:
  \begin{equation}
    \begin{tikzpicture}[diagram,baseline=(A.base)]\label[diagram]{diag:univalence:2}
      \node (X) {$\brk{\beta}$};
      \node[below right = 0.75cm and 2cm of X] (iY) {$\brk{\alpha}$};
      \node[below right = 2.5cm and 1cm of X] (A) {$A$};
      \node[right = 4cm of A] (B) {$B$};
      \node[right = 4cm of iY] (Y) {$\brk{\bar\alpha}$};
      \draw[>->] (A) to (B);
      \draw[->] (X) to node[upright desc] {$w$} (iY);
      \draw[fibration] (X) to (A);
      \draw[fibration] (iY) to (A);
      \draw[fibration] (Y) to (B);
      \draw[>->] (iY) to node [upright desc] {$g$} (Y);
    \end{tikzpicture}
  \end{equation}

  Using \cref{lem:eep}, we can complete \cref{diag:univalence:2} as follows:
  \begin{equation}\label[diagram]{diag:univalence:3}
    \begin{tikzpicture}[diagram,baseline=(A.base)]
      \node (X) {$\brk{\beta}$};
      \node[below right = 0.75cm and 2cm of X] (iY) {$\brk{\alpha}$};
      \node[below right = 2.5cm and 1cm of X] (A) {$A$};
      \node[magenta,right = 4cm of X] (Xext) {$\brk{\bar\beta}$};
      \node[right = 4cm of A] (B) {$B$};
      \node[right = 4cm of iY] (Y) {$\brk{\bar\alpha}$};
      \draw[>->] (A) to (B);
      \draw[->] (X) to node[upright desc] {$w$} (iY);
      \draw[fibration] (X) to (A);
      \draw[fibration] (iY) to (A);
      \draw[fibration] (Y) to (B);
      \draw[>->] (iY) to node [upright desc] {$g$} (Y);
      \draw[fibration,color=LightGray] (Xext) to (B);
      \draw[exists,->,magenta] (Xext) to node[upright desc] {$\bar{w}$} (Y);
      \draw[magenta, exists,>->] (X) to node [above] {$f$} (Xext);
    \end{tikzpicture}
  \end{equation}

  By \textbf{(U8)} we solve the following realignment problem to obtain an
  extension of the code $\Mor[\beta]{A}{U\Sub{\mathcal{U}_0}}$ along
  $\Mor|>->|{A}{B}$, using the fact that $\brk{\bar\beta}$ lies in
  $\mathcal{U}_0$ by assumption:
  \begin{equation}\label[diagram]{diag:univalence:4}
    \begin{tikzpicture}[diagram]
      \node (nw) {$\brk{\beta}$};
      \node (ne) [right = of nw] {$\pi\Sub{\mathcal{U}_0}$};
      \node (sw) [below = of nw] {$\brk{\bar\beta}$};
      \draw[>->] (nw) to node [left] {$f$} (sw);
      \draw[->] (nw) to node [above] {$\beta$} (ne);
      \draw[->,exists,magenta] (sw) to node [sloped,below] {$\bar\beta$ }(ne);
    \end{tikzpicture}
  \end{equation}

  The indicated lift of \cref{diag:univalence:4} then supplies in conjunction
  with the weak equivalence $\Mor[\bar{w}]{\brk{\bar\beta}}{\brk{\bar\alpha}}$
  the required lift for \cref{diag:univalence:1}:
  \[
    \begin{tikzpicture}[diagram]
      \SpliceDiagramSquare{
        width = 4cm,
        height = 2.5cm,
        nw = A,
        sw = B,
        ne = \Eq{E\Sub{\mathcal{U}_0}},
        se = U\Sub{\mathcal{U}_0},
        south = \bar{\alpha},
        north = \prn{\beta, \alpha, w},
        east = \partial_1,
        west/style = >->,
        east/style = fibration,
      }
      \draw[exists,->,magenta] (sw) to node [desc] {$\prn{\bar\beta,\bar\alpha,\bar{w}}$} (ne);
    \end{tikzpicture}
  \]

  Therefore $\partial_1$ is a trivial fibration and thus $\pi\Sub{\mathcal{U}_0}$ is univalent.
\end{proof}

\subsection{Artin gluing and synthetic Tait computability}\label{sec:realignment-stc}

Artin gluing is used by computer scientists to prove metatheorems for type
theories and programming languages such as normalization, canonicity,
decidability, parametricity, conservativity, and computational adequacy.
\citet{sterling-harper:2021} have introduced \emph{synthetic Tait
computability}  as an abstraction for working in the internal language of
glued topoi, taking the realignment law \textbf{(U8)} in its internal
form (see \cref{sec:internal-formulations}) as a basic axiom.

\subsubsection{History and motivation}

Synthetic Tait computability (or \textbf{STC}) was first employed in \opcit to
prove a generalized abstraction/parametricity theorem for a language of
software packages (``modules'') in the style of Standard ML; subsequently,
\citet{sterling-angiuli:2021} used STC to positively resolve the long-standing
\emph{normalization conjecture} for cubical type
theory~\citep{abcfhl:2021}.\footnote{See also Sterling's
dissertation~\citep{sterling:2021:thesis} for a more detailed treatment of both this result and synthetic Tait computability in general.}
Building on these results, \citet{gratzer:2022:lics} adapted STC to
verify the analogous conjecture for \emph{multimodal type
theory}~\citep{gratzer-kavvos-nuyts-birkedal:2020}. In their original formulation, all of these results relied
heavily on \textbf{(U8)}, but the glued topoi in the cited results were all of
presheaf type and hence the presheaf-theoretic universes of
\citet{hofmann-streicher:1997} could be brought to bear without broaching the
question of strict universes in sheaf topoi.

More recently, synthetic Tait computability has been employed in scenarios
where the glued topos is not known to be of presheaf type. For example,
\citet{gratzer-birkedal:2022} proved a canonicity result for a version of
\emph{guarded dependent type theory} for which the necessary instance of STC
involved a Grothendieck topology. It has therefore become a matter
of some urgency to verify the existence of universes satisfying \textbf{(U1-8)}
in arbitrary Grothendieck topoi.

\subsubsection{Universes in Artin gluings}

Let $\Mor[F]{\ECat}{\FCat}$ be a left exact functor between topoi such that
$\ECat$ carries the structure of a model of Martin-L\"of type theory, \ie a
pre-universe $\mathcal{T}$ in the sense of \cref{def:pre-universe}. Write
$\GCat\coloneqq\FCat\downarrow F$ for the Artin gluing of $F$, and let
$\Mor|open immersion|[j]{\ECat}{\GCat}$ be the corresponding open immersion of
topoi. Fixing a universe $\Cls$ in $\GCat$ (\ie a class of maps satisfying
\textbf{(U1--7)}) that contains $j_*\mathcal{T}$, we may define a new
pre-universe $\mathcal{U}$ consisting of the subclass of $\Cls$ spanned by maps
$f$ with $j^*f\in\mathcal{T}$.

We wish to verify that $\mathcal{U}$ likewise carries the structure of a model
of Martin-L\"of type theory in the same sense of satisfying \textbf{(U1,3--5)};
results of this kind are used to prove important syntactic metatheorems for
type theories, such as canonicity (a type theoretic analogue to the existence
property), normalization, decidability of judgmental equality, and
conservativity.

\begin{lemma}
  The class of maps $\mathcal{U}\subseteq\Hom[\GCat]$ satisfies \textbf{(U1,3,4)}.
\end{lemma}

\begin{proof}
  This is a straightforward consequence of the fact that $j^*$ is a logical
  functor, using the fact that $\mathcal{T}$ and $\Cls$ satisfy
  $\textbf{(U1,3,4)}$.
\end{proof}

To show that $\mathcal{U}$ is a pre-universe it remains to verify
\textbf{(U5)}, \ie show that $\mathcal{U}$ has a generic family. It will turn
out that the most elegant way to achieve this factors through an additional
assumption that $\Cls$ satisfies the realignment property \textbf{(U8)}.

\begin{construction}\label[construction]{con:gluing:generic-family}
  We begin by constructing a \emph{putative} generic family for $\mathcal{U}$
  in $\GCat$, which we will subsequently verify to be generic as an application
  of the realignment property for $\Cls$.
  Because $j_*\mathcal{T}\subseteq \Cls$, we have in particular a cartesian
  morphism $\Mor{j_*\pi\Sub{\mathcal{T}}}{\pi\Sub{\Cls}}$ in $\ArrCat{\GCat}$; restricting into the
  open subtopos, we have $\Mor{\pi\Sub{\mathcal{T}}\cong
  j^*j_*\pi\Sub{\mathcal{T}}}{j^*\pi\Sub{\Cls}}$ in $\ArrCat{\ECat}$; writing
  $\Mor[q]{U\Sub{\mathcal{T}}}{j^*U\Sub{\Cls}}$ for the base of this morphism, we
  may define the base of a putative generic family for $\mathcal{U}$ by cartesian
  lift in the gluing fibration:
  \begin{equation}
    \begin{tikzpicture}[diagram,baseline=(sq/sw.base)]
      \SpliceDiagramSquare<sq/>{
        ne = U\Sub{\Cls},
        nw = U\Sub{\mathcal{U}},
        se = j^*U\Sub{\Cls},
        sw = U\Sub{\mathcal{T}},
        nw/style = pullback,
        north/style = {exists,->},
        west/style = {lies over,exists},
        east/style = {lies over},
        north = \bar{q},
        south = q,
      }
      \node (ne) [right = of sq/ne] {$\GCat$};
      \node (se) [right = of sq/se] {$\ECat$};
      \path[fibration] (ne) edge node [upright desc] {$j^*$} (se);
    \end{tikzpicture}
  \end{equation}

  The remainder of the family is defined by pullback:
  \begin{equation}
    \begin{tikzpicture}[diagram,baseline=(sq/sw.base)]
      \SpliceDiagramSquare<sq/>{
        nw/style = pullback,
        north/style = {exists,->},
        west/style = {lies over,exists},
        east/style = {lies over},
        ne = \pi\Sub{\Cls},
        se = U\Sub{\Cls},
        sw = U\Sub{\mathcal{U}},
        nw = \pi\Sub{\mathcal{U}},
        south = \bar{q},
      }
      \node (ne) [right = of sq/ne] {$\ArrCat{\GCat}$};
      \node (se) [right = of sq/se] {$\GCat$};
      \path[fibration] (ne) edge node [upright desc] {$\Cod$} (se);
    \end{tikzpicture}
  \end{equation}
\end{construction}

\subsubsection{An abortive attempt at genericity}\label{sec:gluing:failed-genericity}

Prior to verifying that \cref{con:gluing:generic-family} gives rise to a generic
family for $\mathcal{U}$ under the assumption of realignment for $\Cls$ in \cref{sec:gluing:genericity-via-realignment} below, it is useful to understand intuitively why realignment is needed. Fixing a morphism $\Mor[f]{X}{Y} \in \mathcal{U}$, we wish to construct a cartesian map $\Mor{f}{\pi\Sub{\mathcal{U}}}$. By definition, we have $f \in \Cls$ and
$j^*f \in \mathcal{T}$, hence there exist a pair of cartesian morphisms
$\Mor[x']{f}{\pi\Sub{\Cls}}$ and
$\Mor[x_0]{j^*f}{\pi\Sub{\mathcal{T}}}$. Naively, we might hope to take
advantage of the universal property of $U\Sub{\mathcal{U}}$ \emph{qua} cartesian lift to obtain a
cartesian map $\Mor{f}{\pi\Sub{\mathcal{U}}}$:
\begin{equation}\label[diagram]{diag:gluing:failed-cart}
  \begin{tikzpicture}[diagram,baseline=(sq/sw.base)]
    \SpliceDiagramSquare<sq/>{
      ne = U\Sub{\Cls},
      nw = U\Sub{\mathcal{U}},
      se = j^*U\Sub{\Cls},
      sw = U\Sub{\mathcal{T}},
      nw/style = pullback,
      west/style = {lies over},
      east/style = {lies over},
      north = \bar{q},
      south = q,
      north/node/style = upright desc,
      width = 2.5cm,
    }
    \node (ne) [right = of sq/ne] {$\GCat$};
    \node (se) [right = of sq/se] {$\ECat$};
    \path[fibration] (ne) edge node [upright desc] {$j^*$} (se);
    \node (Pi/U) [left = 2.5cm of sq/nw] {$Y$};
    \node (Pi/T) [left = 2.5cm of sq/sw] {$j^*Y$};
    \path[lies over] (Pi/U) edge (Pi/T);
    \path[->] (Pi/T) edge node [below] {$x_0$} (sq/sw);
    \path[->,color=RedDevil,bend left=35] (Pi/U) edge node [sloped,above,color=black] {$x'$} (sq/ne);
  \end{tikzpicture}
\end{equation}

Unfortunately the configuration of \cref{diag:gluing:failed-cart} is not valid: we
do not have $j^*{x'} = q \circ x$.  If $\Cls$ satisfies \textbf{(U8)},
however, we may choose a \emph{different} upper map $\Mor{Y}{U\Sub{\Cls}}$ that
makes the analogous configuration commute.

\subsubsection{Genericity via realignment}\label{sec:gluing:genericity-via-realignment}

Now we assume that $\Cls$ satisfies the realignment axiom \textbf{(U8)}, and
continue under the same assumptions as \cref{sec:gluing:failed-genericity} to
verify that \cref{con:gluing:generic-family} exhibits a generic family for
$\mathcal{U}$.

\begin{proof}
  We will employ the following realignment in which the upper map is defined by
  adjoint transpose in $j_!\dashv j^*$, and the left-hand map is a monomorphism
  because $j_!j^*E \cong j_!\TermObj{\ECat}\times E$ by Frobenius reciprocity and
  $j_!$ preserves subterminals:
  \begin{equation}\label[diagram]{diag:pi-realignment}
    \begin{tikzpicture}[diagram,baseline=(sw.base)]
      \node (nw) {$j_!j^*f$};
      \node (sw) [below = 2.5cm of nw] {$f$};
      \node (ne) [right = 3.5cm of nw] {$\pi\Sub{\Cls}$};
      \path[>->] (nw) edge node [left] {$\epsilon$} (sw);
      \path[->] (nw) edge node [above] {$\prn{q \circ x_0}^\sharp$} (ne);
      \path[->,exists] (sw) edge node [sloped,below] {$x$} (ne);
    \end{tikzpicture}
  \end{equation}

  \emph{Remark.} To see that the upper and left-hand maps are cartesian, we recall from
  \citet[Proposition 7.7.1]{taylor:1999} that the left adjoint $j_!\dashv j^*$
  creates non-empty limits and the counit
  $\Mor[\epsilon]{j_!j^*}{\IdArr{\GCat}}$ is a cartesian natural
  transformation, \ie its naturality squares are cartesian; these facts follow
  immediately from the strictness of the initial object in the closed subtopos
  $\FCat$. Hence the transpose of a cartesian square from $\ECat$ under the
  adjunction $j_!\dashv j^*$ is a cartesian square in $\GCat$.

  It is a consequence of the commutativity of \cref{diag:pi-realignment} that $x$
  lies over $x_0$:
  \[
    j^*\prn{x}
    = j^*\prn{x \circ \epsilon} \circ \eta
    = j^*\prn{q \circ x_0}^\sharp \circ \eta
    = q \circ x_0
  \]

  We may use the base $\Mor[x]{Y}{U\Sub{\Cls}}$ of the glued morphism from
  \cref{diag:pi-realignment} to extend $x_0$ to $\mathcal{U}$ as desired,
  repairing our failed attempt from \cref{diag:gluing:failed-cart}:
  \begin{equation*}
    \begin{tikzpicture}[diagram,baseline=(sq/sw.base)]
      \SpliceDiagramSquare<sq/>{
        ne = U\Sub{\Cls},
        nw = U\Sub{\mathcal{U}},
        se = j^*U\Sub{\Cls},
        sw = U\Sub{\mathcal{T}},
        nw/style = pullback,
        west/style = {lies over},
        east/style = {lies over},
        north = \bar{q},
        south = q,
        north/node/style = upright desc,
        width = 2.5cm,
      }
      \node (ne) [right = of sq/ne] {$\GCat$};
      \node (se) [right = of sq/se] {$\ECat$};
      \draw[fibration] (ne) to node [upright desc] {$j^*$} (se);
      \node (Y) [left = 2.5cm of sq/nw] {$Y$};
      \node (jY) [left = 2.5cm of sq/sw] {$j^*Y$};
      \draw[lies over] (Y) to (jY);
      \draw[->,exists] (Y) to node [upright desc] {$\exists!$} (sq/nw);
      \draw[->] (jY) to node [below] {$x_0$} (sq/sw);
      \draw[->,color=RegalBlue,bend left=35] (Pi/U) to node [sloped,above,color=black] {$x$} (sq/ne);
    \end{tikzpicture}
    \qedhere
  \end{equation*}
\end{proof}

In fact, this construction above gives a slightly stronger result than \textbf{(U5)}.

\begin{theorem}
  \label[theorem]{thm:realign-at-syntax}
  Given $\Mor[f]{X}{Y} \in \mathcal{U}$ together with a cartesian map
  $\Mor[x_0]{j^*f}{\pi\Sub{\mathcal{T}}}$, there exists a cartesian map
  $\Mor[x]{f}{\pi\Sub{\mathcal{U}}}$ lying over $x_0$:
  \begin{equation*}
    \begin{tikzpicture}[diagram,baseline=(sw.base)]
      \SpliceDiagramSquare<sq/>{
        width = 3cm,
        nw/style = pullback,
        nw = f,
        ne = \pi\Sub{\mathcal{U}},
        sw = j^*f,
        north = x,
        south = x_0,
        se = \pi\Sub{\mathcal{T}},
        west/style = lies over,
        east/style = lies over,
        north/style = {exists,->},
      }
      \node (ne) [right = of sq/ne] {$\ArrCat{\GCat}$};
      \node (se) [right = of sq/se] {$\ArrCat{\ECat}$};
      \path[fibration] (ne) edge node [upright desc] {$\prn{j^*}^\to$} (se);
    \end{tikzpicture}
  \end{equation*}
\end{theorem}

This property is particularly useful in proofs of metatheorems of type theories
based on Artin
gluing~\citep{sterling-angiuli-gratzer:2022,sterling-angiuli:2021,gratzer:2022:lics}.
In this context, one typically requires not only that $\mathcal{U}$ be a
pre-universe, but that the chosen codes witnessing \textbf{(U3,4)} are
moreover preserved by $j^*$. Without \cref{thm:realign-at-syntax}, these strict
equations would preclude a conceptual construction of these codes.

\begin{remark}
  \Citet{uemura:2017} presents an alternative construction for a pre-universe in
  $\GCat$ satisfying \cref{thm:realign-at-syntax}. Rather than relying on
  \textbf{(U8)}, \citeauthor{uemura:2017} begins with separate pre-universes from $\ECat$ and
  $\FCat$ and combines them directly. This explicit decomposition ensures that
  the resultant universe satisfies the special case of \textbf{(U8)} necessary
  for \cref{thm:realign-at-syntax}.
\end{remark}

\section{Conclusions and future work}
\label{sec:conclusion}

We have shown that every Grothendieck topos can be equipped with a cumulative
hierarchy of universes satisfying \textbf{(U1--8)} assuming sufficient
universes in the background set theory. This result is important because it
extends the Hofmann--Streicher interpretation of Martin-L\"of type theory in
presheaf topoi to arbitrary sheaf topoi.

\subsection{Prospects for a constructive version}\label{sec:conclusion:constructive}

Our constructions are highly classical; in particular, we rely on the theory of
locally presentable categories and $\kappa$-compactness, both of which make
heavy use of choice. Developing a constructively acceptable version of
\cref{sec:universe} remains an open problem. We briefly survey the landscape of
universes within a particular constructive metatheory: the internal language of
an elementary topos $\ECat$.

Although the literal definition of a Grothendieck universe is meaningless in
$\ECat$, we can proceed analogously and fix a generic map $\Mor{\VEL}{\VTY}$
satisfying the appropriate version of \textbf{(U2--4,6)}. The class $\Cls_\VTY$
classified by this map then satisfies \textbf{(U1--6)}. Already some care must
be taken; without choice, a family with $\VTY$-small fibers need not be
classified by a map into $\VTY$. Absent the law of the excluded middle,
\textbf{(U8)} is satisfied for at least the class of
\emph{decidable} monomorphisms $\Mor|>->|{A}{B}$.

The Hofmann--Streicher construction exposed in \cref{sec:hs-and-s} works over $\ECat$ without
modification. In particular, the standard generic
family of $\Cls_\VTY$ lifts to a universe in the category of internal presheaves
$\Psh[\ECat]{\CCat}$ for any $\VTY$-small internal category $\CCat$. The class of maps
$\PrCls\Sub{\VTY}$ classified by this map satisfies
\textbf{(U1--6)}. \textbf{(U8)} is satisfied only for the class of
\emph{level-wise} decidable monomorphisms: monomorphisms $\Mor|>->|{A}{B}$ whose components
$\Mor|>->|{A\prn{c}}{B\prn{c}} \in \Hom[\ECat]$ are all decidable~\citep{orton-pitts:2016}. In fact, \citet{swan:2018} shows that this
result is sharp: it is possible to choose a base topos in such a way that this generic map cannot satisfy \textbf{(U8)} for all
monomorphisms, though it remains possible that there is \emph{another} generic map satisfying
\textbf{(U8)} for all monomorphisms.
Finally, this universe induces a universe $\ShCls\Sub{\VTY}$ in any sheaf subtopos
$\Sh[\ECat]{\CCat,J}$. The construction is identical to that of \cref{sec:hs-and-s} and $\ShCls\Sub{\VTY}$ satisfies \textbf{(U1--6)} just as in the classical setting. In
this setting, however, the status of \textbf{(U8)} remains entirely open for
this universe.

Over a base topos $\ECat$ not satisyfing the axiom of choice, it is reasonable to hope
that properties such as \textbf{(U7)} or \textbf{(U8)} might lift from $\ECat$
to any topos bounded over $\ECat$; this lifting is verified for \textbf{(U7)}
in the context of \emph{algebraic set
theory}~\citep{joyal-moerdijk:1995,van-den-berg:habil}, but the corresponding
lifting for \textbf{(U8)} remains a conjecture.

\printbibliography

\end{document}